\documentclass[11pt, USletter, reqno, twoside, makeidx]{amsart}
\usepackage[margin=2.8cm, marginparsep=0.2cm, marginparwidth=2.4cm]{geometry}
\usepackage[color=blue!20!white,textsize=tiny]{todonotes}
\usepackage[font={small,it}]{caption}
\usepackage{amsmath} 
\usepackage{amsthm}
\usepackage{amssymb} 
\usepackage{amsfonts}\usepackage{tikz-cd}
\usepackage{graphicx}
\usepackage{xcolor}
\definecolor{darkgreen}{rgb}{0,0.5,0}
\usepackage[normalem]{ulem}
\usepackage{comment}
\usepackage{xypic} 
\usepackage[all]{xy}
\usepackage{tikz-cd}
\usepackage[autostyle]{csquotes}
\usepackage[makeroom]{cancel}
\usepackage{braket}

\usepackage{caption,subcaption,pinlabel}
\usepackage{graphics}
\usepackage{hyperref}

\usepackage{comment}
\usepackage{enumitem}
\usepackage{mathtools}
\usepackage{tikz}
\usetikzlibrary{matrix,arrows,decorations.pathmorphing}
\usepackage{mathabx}
\usetikzlibrary{matrix,calc}
\usepackage{scrextend}

\newtheorem{theorem}{Theorem}[section]
\newtheorem{lemma}[theorem]{Lemma}
\newtheorem{proposition}[theorem]{Proposition}

\newtheorem{corollary}[theorem]{Corollary}

\theoremstyle{definition}
\newtheorem{remark}[theorem]{Remark}
\newtheorem{definition}[theorem]{Definition}
\newtheorem{example}[theorem]{Example}

\def\im{\text{im }}
\def\spinc{\text{spin}^c}

\newcommand{\Z}{\mathbb{Z}}

\newcommand{\C}{\mathbb{C}}

\newcommand{\RP}{\mathbb{R} \mathbb{P}}

\newcommand{\Char}{\text{Char}}

\newcommand{\Hom}{\text{Hom}}
\newcommand{\CPr}{\mathbb{CP}}

\def\HF {\mathit{HF}}
\newcommand\HFhat{\widehat{\mathit\HF}}

\newcommand\HFp {\HF^+}
\newcommand\HFpe {\HF^+_\text{even}}
\newcommand\HFpo {\HF^+_\text{odd}}

\newcommand{\Hl}{\mathbb{H}}
\newcommand{\Hlb}{\overline{\mathbb{H}}}
\newcommand{\Is}{\smash{I^\#}}
\newcommand{\Iseven}{\smash{I^\#_\text{even}}}
\newcommand{\Isodd}{\smash{I^\#_\text{odd}}}
\newcommand{\Ts}{\smash{T^\#}}
\newcommand{\Tsb}{\smash{\bar{T}^\#}}

\newcommand{\ev}[2]{#1(#2)}

\begin{document}
\title[]{Instanton Floer homology of almost-rational plumbings}
\author{Antonio Alfieri, John A. Baldwin, Irving Dai, and Steven Sivek }
\begin{abstract}
We show that if $Y$ is the boundary of an almost-rational plumbing, then the framed instanton Floer homology $\smash{I^\#(Y)}$ is isomorphic to the Heegaard Floer homology $\smash{\widehat{\mathit{HF}}(Y; \C)}$. This class of 3-manifolds includes all Seifert fibered rational homology spheres with base orbifold $S^2$ (we establish the isomorphism for the remaining Seifert fibered rational homology spheres---with base  $\mathbb{RP}^2$---directly). Our proof utilizes lattice homology, and relies on a decomposition theorem for instanton Floer cobordism maps recently established by Baldwin and Sivek.   
\end{abstract}
\maketitle 

\section{Introduction}\label{sec:1}

Instanton Floer homology is a gauge-theoretic invariant which has featured prominently in many important topological results, including the proof of the Property~P conjecture for knots \cite{KMpropertyP}, and the proof that Khovanov homology detects the unknot \cite{KMkhovanov}. An active area of research involves developing techniques for computing  instanton Floer homology, and using these to explore  relationships between the instanton theory and other Floer homological invariants. 

In this paper, we study a version of instanton Floer homology called \textit{framed instanton homology}, introduced by Kronheimer and Mrowka \cite{KMknothomology} and further developed by Scaduto \cite{Scaduto}. A conjecture of Kronheimer and Mrowka \cite[Conjecture 7.24]{KMknots} posits that the framed instanton homology $\Is(Y)$ of a closed $3$-manifold $Y$ is isomorphic to the Heegaard Floer homology $\smash{\widehat{\mathit{HF}}(Y; \C)}$. Thus far, framed instanton  homology has  been computed (and this conjecture verified) for: connected sums of $S^1 \times S^2$ \cite{Scaduto}, branched double covers of two-fold quasi-alternating links \cite{Scaduto, ScadutoStoffregen}, some Brieskorn spheres \cite{Scaduto}, a surface times a circle \cite{ChenScaduto}, various Dehn surgeries on knots \cite{BSLspace, LPCS, BSconcordance}, and some small-volume hyperbolic manifolds \cite{BSconcordance}.

The aim of our work is to substantially expand the class of 3-manifolds for which the framed instanton homology is known, by establishing (under favorable circumstances) a relationship between $\Is$ and \emph{lattice homology}. The latter is a combinatorial invariant of plumbed 3-manifolds,\footnote{Associated to  \emph{negative-definite} plumbings.} introduced by Ozsv\'ath and Szab\'o \cite{OSplumbed} and elaborated considerably by N\'emethi \cite{Nemethi}, which is conjecturally isomorphic to Heegaard Floer homology. This conjecture has been verified for the class of \emph{almost-rational} plumbings, which includes\footnote{Up to orientation reversal.} all Seifert fibered rational homology spheres with base  $S^2$.  Our main result is that a version of lattice homology is likewise isomorphic to framed instanton homology in even degree for almost-rational plumbings. As a corollary, we prove that Kronheimer and Mrowka's conjectured isomorphism \[\smash{\Is(Y) \cong \HFhat(Y; \C)}\] holds for  such plumbings.  We note that almost-rational plumbings and lattice homology generally are also interesting from the viewpoint of algebraic geometry; see \cite{Nemethi}.

Our approach is modeled after  analogous arguments in the Heegaard Floer setting,  in \cite{OSplumbed, Nemethi}. However, there are several additional challenges in the instanton Floer setting. First, the instanton groups and cobordism maps come equipped with decorations which must be taken into account---namely, multicurves and surfaces which specify bundles over the respective manifolds (see Section~\ref{sec:2.1}). Second, we need to establish ``adjunction relations" for framed instanton homology, analogous to those in Heegaard Floer homology \cite{OSplumbed}. Our proof of these adjunction relations uses  Baldwin and Sivek's recent decomposition result for cobordism maps in framed instanton homology \cite{BSLspace}. Third, there are standard properties of $\HF^+$, $\HF^-$, and $\HF^\infty$ which play key roles in the Heegaard Floer argument, whose analogues in framed instanton  homology have not yet been established. Hence, considerable care and some new results are needed when adapting the arguments of \cite{OSplumbed,Nemethi} to our setting.

We describe our main results more precisely below.

\subsection{Statement of results}\label{sec:1.1}
Suppose $\Gamma$ is a negative-definite plumbing forest, and let $Y_\Gamma$ denote the boundary of the associated plumbed 4-manifold.  In Section~\ref{sec:3.1}, we define a complex vector space $\Hl(\Gamma)$ called the \textit{lattice homology of $\Gamma$}.  We will be particularly interested in the case where $\Gamma$ is \emph{almost-rational}  (see Section~\ref{sec:2.4}), as indicated in our main theorem:

\begin{theorem}\label{thm:1.1}
If $\Gamma$ is  almost-rational, then  \[\Hl(\Gamma) \cong \Iseven(Y_\Gamma)\subset \Is(Y_\Gamma),\] where $\Iseven(Y_\Gamma)$ refers to the summand of $\Is(Y_\Gamma)$ in even $\Z/2\Z$-grading.
\end{theorem}

The isomorphism of Theorem~\ref{thm:1.1} is effected by an explicit map, defined in Section~\ref{sec:3.2}. 

Scaduto proved in \cite[Corollary 1.4]{Scaduto} that for any rational homology sphere $Y$, the  Euler characteristic of $\Is(Y)$ is given by
\begin{equation}\label{eq:1.1}
\chi(\Is(Y)):= \dim \Iseven(Y) - \dim \Isodd(Y)= |H_1(Y)|.
\end{equation}
Therefore, Theorem~\ref{thm:1.1} suffices to compute the \emph{entire} framed instanton homology group $\Is(Y_\Gamma)$ as a $\Z/2\Z$-graded complex vector space.  An isomorphism similar to that in Theorem \ref{thm:1.1} holds in the Heegaard Floer setting, due to Ozsv\'ath--Szab\'o \cite[Theorem 1.2]{OSplumbed} and N\'emethi \cite[Theorem 5.2.2]{Nemethi}. By comparing the lattice homology theory defined here with that in the Heegaard Floer setting, we  then obtain the following:

\begin{corollary}\label{cor:1.2}
If $\Gamma$ is almost-rational, then there is an isomorphism 
\[
\Is(Y_\Gamma) \cong \HFhat(Y_\Gamma; \C)
\]
of $\Z/2\Z$-graded complex vector spaces.
\end{corollary}

As noted previously, the class of $3$-manifolds specified by almost-rational plumbings  includes all Seifert fibered rational homology spheres with base  $S^2$ (after possibly reversing orientation).  Every other Seifert fibered rational homology sphere has base  $\mathbb{RP}^2$, and we mimic an argument of Boyer--Gordon--Watson in the Heegaard Floer setting \cite[Proposition 18]{BGW} in this case to prove that each such manifold is an instanton L-space. We therefore have:

\begin{corollary}\label{cor:1.2a}
If $Y$ is a Seifert fibered rational homology sphere, then there is an isomorphism
\[
\Is(Y) \cong \HFhat(Y; \C)
\]
of $\Z/2\Z$-graded complex vector spaces.
\end{corollary}

We can then apply this to prove an instanton version of the Boyer--Gordon--Watson L-space conjecture \cite{BGW} for Seifert fibered rational homology spheres:\begin{corollary}\label{cor:1.3}
Let $Y$ be a Seifert fibered rational homology sphere. Then $Y$ is an instanton L-space if and only if its fundamental group is not left-orderable.
\end{corollary}


An important advantage of instanton Floer homology over other Floer theories is that the relationship between Floer homology and the fundamental group is much more explicit in the instanton setting.  In \cite[Theorem 4.6]{BSstein}, for instance, the authors prove that if $Y$ is not an instanton L-space, and certain nondegeneracy conditions are satisfied, then there must exist an irreducible representation $\pi_1(Y) \rightarrow SU(2)$. These nondegeneracy conditions are always satisfied if, for example, the universal abelian cover of $Y$ is a rational homology sphere. We therefore immediately obtain:

\begin{corollary}\label{cor:1.4}
Suppose $\Gamma$ is almost-rational. Suppose that $Y_\Gamma$ is not a Heegaard Floer L-space and   the universal abelian cover of $Y_\Gamma$ is a rational homology sphere. Then there is an irreducible representation $\pi_1(Y_\Gamma) \rightarrow SU(2)$.
\end{corollary}

In \cite[Theorem 1.3]{Pedersen}, Pedersen provides a combinatorial procedure for determining when the universal abelian cover of a graph manifold is a rational homology sphere. Thus, if $\Gamma$ is almost-rational, then a computation of $\Hl(\Gamma)$ (in conjunction with Pedersen's work) can be used in principle to deduce the existence of irreducible $SU(2)$-representations of $\pi_1(Y_\Gamma)$. We illustrate this principle with an example in Section~\ref{sec:8}.

As mentioned earlier, our proof of Theorem~\ref{thm:1.1} requires the following ``adjunction relations" for cobordism maps in framed instanton homology, which we establish in Section \ref{sec:4} and may be of independent interest:

\begin{theorem}\label{thm:1.5}
Let $(W,\nu)$ be a cobordism \[(W, \nu): (Y_1,\lambda_1)\to (Y_2,\lambda_2)\] with $b_1(W) = 0$, such that at least one of $Y_1$ or $Y_2$ is a rational homology sphere.  Let $\Sigma$ be an embedded 2-sphere in $W$ of (negative) self-intersection $-p$, and let \[s : H_2(W; \Z) \rightarrow \Z\] be a homomorphism. Then the cobordism map \[\Is(W, \nu; s):\Is(Y_1,\lambda_1) \to \Is(Y_2,\lambda_2)\] satisfies the following identities:
\begin{enumerate}
\item If $|\ev{s}{[\Sigma]}| > p$, then $\Is(W, \nu; s) = 0$.
\item If $\ev{s}{[\Sigma]} = \pm p$, then $\Is(W, \nu; s) = (-1)^{\phantom{.}\nu \cdot \Sigma\phantom{.}} \Is(W, \nu; s \pm 2[\Sigma]^*)$.
\end{enumerate}
Here, $\ev{s}{[\Sigma]}$ denotes the evaluation of $s$ on $[\Sigma]$, and $[\Sigma]^*$ is the Poincar{\'e} dual of $[\Sigma]$.
\end{theorem}

\subsection{Organization}
In Section~\ref{sec:2}, we review some formal properties of framed instanton  homology and establish some terminology related to plumbings. In Section~\ref{sec:3}, we define our version of lattice homology and the associated map to framed instanton  homology. We prove that this map is  well-defined in Section~\ref{sec:4}, where we also prove the adjunction relations  of Theorem~\ref{thm:1.5}. In Section~\ref{sec:5}, we review the construction of the surgery exact sequence in lattice homology. This is used to begin the proof of Theorem~\ref{thm:1.1} in Section~\ref{sec:6}. We complete the proof in Section~\ref{sec:7}, where we also establish Corollaries~\ref{cor:1.2}, \ref{cor:1.2a}, and \ref{cor:1.3}. In Section~\ref{sec:8}, we discuss some examples and applications. Finally, in Appendix~\ref{sec:appendix} we prove a technical result about the invariance of lattice homology under blowing down leaves with framing $-1$.

\subsection{Acknowledgements} Irving Dai was supported by NSF Grant DMS-1902746. John Baldwin was supported by NSF Grant DMS-1454865. We are grateful to Juanita Pinz\'on-Caicedo and Liam Watson for motivational conversations, and to the anonymous referee for helpful feedback.

\section{Background}\label{sec:2}

\subsection{Framed instanton homology}\label{sec:2.1} 
Let $Y$ be a closed, connected, oriented 3-manifold, and  $\lambda \subset Y$ an oriented multicurve. (That is, $\lambda$ is an oriented, embedded 1-submanifold of $Y$.) Framed instanton homology assigns to the pair $(Y, \lambda)$ a $\Z/2\Z$-graded, finite-dimensional complex vector space $\Is(Y, \lambda)$. The isomorphism class of this vector space depends only on the diffeomorphism type of $Y$ and the homology class of $\lambda$ in $H_1(Y; \Z/2\Z)$. In this subsection, we discuss the difficulties that arise when defining instanton Floer homology, and explain how this \emph{framed} construction avoids these. The reader should see \cite{Donaldsonbook, KMknothomology, Scaduto} for more details.

The instanton Floer homology of a 3-manifold is defined, roughly, as the homology of a Morse chain complex associated to a Chern-Simons functional. The latter is a non-linear functional  on an infinite-dimensional space of connections associated to a principal $SU(2)$- or $SO(3)$-bundle over the 3-manifold, modulo the action of a corresponding gauge group (see e.g.\ \cite{Donaldsonbook}). The critical points of this functional are  the flat connections up to gauge equivalence. The principal difficulty in defining instanton Floer homology is that flat connections can be reducible, with nontrivial stabilizers under the action of the gauge group (such stabilizers can include subgroups isomorphic to $U(1)$, or, in the $SO(3)$ case, the Klein four-group among others \cite[p.236]{FloerExact}). 
This issue can be avoided by working with bundles which satisfy a certain admissibility condition, as explained below.

Let $E$ be a $U(2)$-bundle over $Y$, and  $L$  a line bundle over $Y$ with a fixed isomorphism \[L\xrightarrow{\cong}\det(E).\] Denote by $\smash{\mathcal{A}_{U(2)}}$ the space of $U(2)$-connections on $E$, and by $\smash{\mathcal{A}_{SO(3)}}$ the space of $SO(3)$-connections on ad$(E)$. Let $\mathcal{A}_{U(1)}$ denote the space of $U(1)$-connections on $L$, and fix a reference connection $\alpha$ therein. Given the natural identification \[\smash{\mathcal{A}_{U(2)} = \mathcal{A}_{SO(3)} \times \mathcal{A}_{U(1)}},\] we may identify $\smash{\mathcal{A}_{SO(3)}}$ with the subset of connections in $\smash{\mathcal{A}_{U(2)}}$ with the fixed determinant connection $\alpha$, \begin{equation}\label{eqn:ident}\smash{\mathcal{A}_{SO(3)} \cong \mathcal{A}_{SO(3)} \times \{\alpha\} \subset \mathcal{A}_{U(2)}}.\end{equation} We say that a connection in this subset is \emph{projectively flat} if its $\mathcal{A}_{SO(3)}$ component is flat. Note that we can view the group of \emph{determinant-1 gauge transformations} of $E$, \[
\mathcal{G}^{(\det=1)}=\{g \in \Gamma(\text{Aut}(E)) \mid \det(g)=1\},
\] as acting on $\smash{\mathcal{A}_{SO(3)}}$ via the identification in \eqref{eqn:ident}.
There is a many-to-one map \[\mathcal{A}_{SO(3)}/\mathcal{G}^{(\det=1)} \to \mathcal{A}_{SO(3)}/\mathcal{G},\] where $\mathcal{G}$ denotes the usual group of $SO(3)$-gauge transformations. 

The advantage of this setup over that of standard $SO(3)$-gauge theory, where one studies the Chern-Simons functional on $\mathcal{A}_{SO(3)}/\mathcal{G}$, is that  reducible connections are easier to understand in the determinant-1 setting: a connection $\nabla\in \mathcal{A}_{SO(3)}$ has nontrivial stabilizer (i.e. larger than $\Z/2\Z$) in the determinant-1 gauge group if and only if $\nabla$ is compatible with a splitting of $E$ into line bundles. If $\nabla$ is moreover projectively flat, then Chern-Weil theory, applied to this splitting, implies that the restriction of $E$ to any surface $\Sigma \subset Y$ has even degree. This leads to the admissibilty condition below.

A pair $(Y,\lambda)$ as above is called \emph{admissible} if the multicurve $\lambda$ intersects some surface $\Sigma \subset Y$ in an odd number of points. To define the instanton Floer homology $I_*(Y,\lambda)$ of this pair, we then choose $E$ so that $c_1(E)$ is Poincar\'e dual to $\lambda$. From the discussion above, this ensures that $E$ admits no reducible, projectively flat connections. The construction of this Floer group can then proceed in the usual manner, via the Morse homology of the Chern-Simons functional on \[\mathcal{A}_{SO(3)}/\mathcal{G}^{(\det=1)}.\] The resulting complex vector space $I_*(Y,\lambda)$ has a relative $\Z/8\Z$-grading which reduces to a canonical absolute $\Z/2\Z$-grading. For further details, see e.g.\ \cite[Section 5.6]{Donaldsonbook}.

\begin{remark} \label{rmk:critCS}The critical points of the (unperturbed) Chern-Simons functional on $\mathcal{A}_{SO(3)}/\mathcal{G}^{(\det=1)}$ are the projectively flat connections, up to determinant-1 gauge. These are in one-to-one correspondence with representations 
\[\rho: \pi_1(Y-\lambda)\to SU(2)\]
such that $\rho(\mu)=-I$ for every meridian $\mu\subset Y-\lambda$, up to conjugation in $SU(2)$.

\end{remark}

The admissibility condition implies that $b_1(Y) > 0$. To define instanton Floer homology for an arbitrary pair $(Y,\lambda)$ of $3$-manifold and multicurve, we apply the \emph{framed} construction from \cite{KMknothomology}, as follows. First, we form the stabilized pair \[
(Y \# \phantom{.}T^3, \lambda \cup \gamma),
\]
where $\gamma$ denotes a fiber of the trivial $S^1$-bundle over $T^2$. Since this pair is admissible, we can define its instanton Floer homology \[I_*(Y \# \phantom{.}T^3, \lambda \cup \gamma)\] as described above. There is a degree-4 involution $\frac{1}{2}\mu(\textrm{pt})$ on this group, and the \emph{framed instanton homology} of $(Y,\lambda)$, denoted by \[\Is(Y,\lambda),\] is the complex vector space given by fixed point set of this involution; in particular, \[\dim \Is(Y,\lambda) = \frac{1}{2}\dim I_*(Y \# \phantom{.}T^3, \lambda \cup \gamma).\] This framed group is relatively $\Z/4\Z$-graded, and retains the absolute $\Z/2\Z$-grading from the Floer homology of $(Y \# \phantom{.}T^3, \lambda \cup \gamma)$.

\begin{remark}
It is not hard to see from Remark \ref{rmk:critCS} that the critical points of the (unperturbed) Chern-Simons functional for the pair $(Y \# \phantom{.}T^3, \lambda \cup \gamma)$ are in two-to-one correspondence with representations \[\rho:\pi_1(Y)\to SU(2),\] where we do \emph{not} mod out by conjugation. 
\end{remark}

\begin{remark}
The isomorphism class of $\Is(Y, \lambda)$ depends only on the isomorphism type of the $SO(3)$-bundle \[\textrm{ad}(E)\to Y \# \phantom{.}T^3,\] which is classified by its second Stiefel-Whitney class. Therefore, up to isomorphism, $\Is(Y, \lambda)$ depends only on the diffeomorphism type of $Y$ and the homology class of $\lambda$ in $H_1(Y; \Z/2\Z)$, as mentioned above. For the naturality required for cobordism maps, however, we will need to keep track the oriented multicurve $\lambda$ itself. This orientation can be viewed as specifying the $U(2)$-lift $E$ of the $SO(3)$-bundle ad$(E)$.
\end{remark}

\subsection{Cobordism maps}\label{sec:2.2}
Let $(Y_0, \lambda_0)$ and $(Y_1, \lambda_1)$ be two 3-manifold and multicurve pairs. Let $W$ be a cobordism from $Y_0$ to $Y_1$, and $\Sigma$ be a properly embedded surface in $W$ with $\partial \Sigma = - \lambda_0 \cup \lambda_1,$ so that $(W,\Sigma)$ is a cobordism of pairs, \[(W,\Sigma): (Y_0, \lambda_0)\to (Y_1, \lambda_1)\] (the surface $\Sigma$ will sometimes be referred to as the \emph{surface decoration} on $W$). As shown in \cite{KMkhovanov}, this cobordism $(W, \Sigma)$ induces a complex-linear map \[\Is(W,\Sigma): \Is(Y_0, \lambda_0)\to \Is(Y_1, \lambda_1).\] According to \cite[Proposition~7.1]{Scaduto}, this map has    mod 2 degree given by 
\begin{equation}\label{eqn:degree} -\frac{3}{2}(\chi(W)+\sigma(W)) + \frac{1}{2}(b_1(Y_1)-b_1(Y_0)) \bmod{2}. \end{equation}
In \cite[Theorem 1.16]{BSLspace}, Baldwin and Sivek showed that these maps satisfy a $\spinc$-like decomposition analogous to that of cobordism maps in Heegaard Floer homology \cite{OSsmooth4}. Specifically, they proved the following:
\begin{theorem}\cite[Theorem 1.16]{BSLspace}\label{thm:decomposition}
Let \[(W,\Sigma): (Y_0, \lambda_0)\to (Y_1, \lambda_1)\] be a cobordism with $b_1(W)=0$. Then $\Is(W,\Sigma)$ is a sum
\[ \Is(W,\Sigma) = \sum_{k: H_2(W; \Z)\to \Z} \Is(W,\Sigma; k) \]
of maps \[\Is(W,\Sigma; k): \Is(Y_0, \lambda_0)\to \Is(Y_1, \lambda_1)\] such that:
\begin{enumerate}
\item (finiteness) $\Is(W,\Sigma; k)=0$ for all but finitely many $k$;
\item (adjunction inequality) If $\Is(W,\Sigma; k)$ is non-zero, then $\ev{k}{x} \equiv x^2 \bmod 2$ for all $x \in H_2(W; \Z)$, and $k$ satisfies the adjunction inequality 
\[
|\ev{k}{[S]}| + [S] \cdot [S] \leq -\chi(S)
\]
for every smoothly embedded, connected surface $S \subset W$ with positive self-intersection and genus at least one;
\item (composition) Let \[(W',\Sigma'): (Y_1, \lambda_1)\to (Y_2, \lambda_2)\] be another cobordism such that $b_1(W')=0$, and let $k' : H_2(W'; \Z) \rightarrow \Z$. Then
\[ 
\Is(W',\Sigma'; k') \circ \Is(W,\Sigma; k) = \sum_{\substack{s: H_2(W \cup W'; \Z)\to \Z \text{ with} \\ s|_W=k \text{ and } s|_{W'}=k'}} \Is(W \cup W', \Sigma \cup \Sigma'; s) 
\]
where $W \cup W'$ denotes the cobordism obtained by gluing $W$ and $W'$ along their common boundary $Y_1$;
\item (blow-up formula) Let $\widetilde{W}=W\# \overline{\CPr}^2$ denote the blow-up of $W$ at a point away from $\Sigma$. Then
\[\Is(\widetilde{W},\Sigma; k+\ell E^*) =
\begin{cases}
\frac{1}{2} \Is(W,\Sigma; k) & \text{ if } \ell = \pm 1 \\
0 & \text{ otherwise } 
\end{cases}
\]
where $E^*$ denotes the Poincar\'e dual to the exceptional divisor;
\item (sign relation) For any $\alpha \in H_2(W; \Z)$, we have
\[
\Is(W, \Sigma + \alpha; k)=(-1)^{\frac{1}{2}(\ev{k}{\alpha} + \alpha \cdot \alpha) + \alpha \cdot [\Sigma]\phantom{.}} \Is(W, \Sigma; k).
\] 
\end{enumerate} 
\end{theorem}
\noindent

\begin{remark}\label{rmk:dependence}
The maps $\Is(W,\Sigma; k)$ depend on $\Sigma$ only up to its class in $H_2(W,\partial W;\Z)$. Up to an overall sign (independent of $k$), these maps depend  on $\Sigma$ only up to its class in $H_2(W,\partial W;\Z/2\Z)$.
\end{remark}

\begin{remark}The framed instanton homology  of $Y$ also admits a decomposition along elements $k:H_2(Y; \Z) \rightarrow 2\Z$ \cite[Corollary 7.6]{KMknots}, and Theorem~\ref{thm:decomposition} is natural with respect to this decomposition \cite[Section 6]{BSLspace}. However, as we will only be concerned with rational homology spheres, this will not be important for our purposes.
\end{remark}


\subsection{The surgery exact triangle}\label{sec:2.3} The surgery exact triangle is
an important feature of many Floer homology theories. It was first established in  instanton Floer homology by Floer \cite{FloerExact}, though the formulation here comes from the work of Scaduto \cite{Scaduto}.

\begin{theorem}[\cite{FloerExact}, \cite{Scaduto}]\label{exact}
Let $Y$ be a closed, oriented 3-manifold, and $\lambda \subset Y$ an oriented multicurve. Fix any (oriented, framed) knot $K \subset Y - \lambda$. Then for any integer $n$, there is an exact triangle
\[
\cdots \rightarrow \Is(Y, \lambda) \xrightarrow{\Is(W_a, \Sigma_a)} \Is(Y_n(K), \lambda \cup \mu_K) \xrightarrow{\Is(W_b, \Sigma_b)} \Is(Y_{n+1}(K), \lambda) \rightarrow \cdots.
\]
Here, $\mu_K$ is a (positively-oriented) meridian of $K$ in $Y$.
\end{theorem}

\begin{remark}
A  sequence of manifolds of the form $Y, Y_n(K), Y_{n+1}(K)$ is called a \emph{surgery triad}.
\end{remark}

We now describe the surface decorations $\Sigma_a, \Sigma_b$ appearing in the cobordism maps above. We begin with some terminology. 

Suppose $\lambda$ is an oriented multicurve in $Y$, and let $K \subset Y-\lambda$ be an oriented knot. Denote by $W_n(K)$ the trace of $n$-framed surgery along $K$; i.e., $W_n(K)$ is obtained from $I\times Y$ by attaching an $n$-framed 2-handle along $\{1\}\times K\subset \{1\}\times Y$, where $I = [0,1]$. Let 
\[
C_{\lambda} = I \times \lambda  \subset I \times Y\subset W_n(K)
\]
denote the \emph{cylinder over $\lambda$}, and let $\mu_K$ be a positively-oriented meridian of $\{1\}\times K$ in $\{1\}\times Y\subset W_n(K)$. Let $D_{\mu_K}$ denote the disk formed by pushing into the interior of $I \times Y$ a meridional disk in $\{1\}\times Y$ bounded by $\mu_K$; see Figure \ref{fig:standard}.

\begin{definition}
The \textit{standard decoration} on $W_n(K)$ is the union $C_{\lambda} \cup D_{\mu_K}$. This decoration  has incoming boundary $\lambda =\{0\}\times \lambda \subset Y$ and outgoing boundary $\lambda \cup \mu_K = (\{1\}\times \lambda) \cup \mu_K\subset Y_n(K)$.
 \end{definition}

\begin{figure}[t]
\center
\includegraphics[scale=0.8]{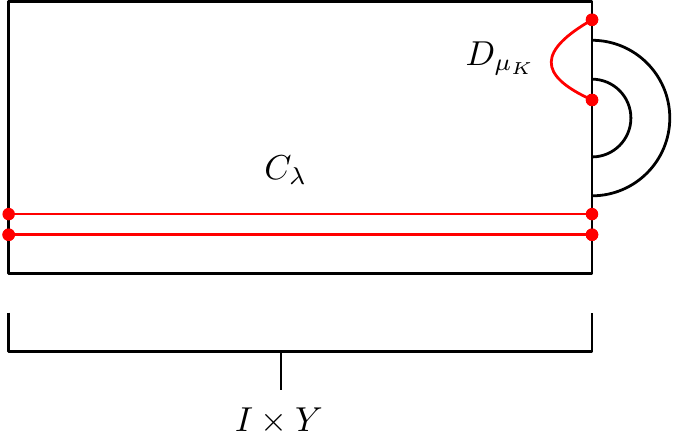}
\caption{Schematic depiction of the standard two-handle decoration.}\label{fig:standard}
\end{figure}

\begin{remark}[Naturality]\label{rem:tautological}
Suppose  $\lambda$ and $\lambda'$ are two disjoint multicurves in $Y$ such that $\lambda'$ is nullhomologous in $Y-\lambda$. Then the groups $\Is(Y, \lambda)$ and $\Is(Y, \lambda \cup \lambda')$ may be identified via the following \textit{tautological cobordism map}:
\[
\Is(I \times Y, C_{\lambda} \cup S_{\lambda'}) : \Is(Y, \lambda) \rightarrow \Is(Y, \lambda \cup \lambda').
\]
Here, $S_{\lambda'}$ denotes any surface obtained by pushing into the interior of $I \times Y$ a Seifert surface for $\{1\}\times\lambda'$ lying in the complement of $\lambda$. We likewise have the reverse of the tautological cobordism map, associated to the cobordism we get by turning $(I \times Y, C_{\lambda} \cup S_{\lambda'})$ upside-down. This map is an isomorphism, since the composition of the tautological cobordism with its reverse is just $I \times Y$ equipped with a decoration differing from $I \times \lambda$ by a closed nullhomologous surface, and the map associated to a cobordism $(W,\nu)$ depends on $\nu$ only up to its homology class in $H_2(W,\partial W)$, per Remark \ref{rmk:dependence}. The tautological cobordism map is similarly independent of the Seifert surface $S_{\lambda'}$, since this surface can be pushed into the boundary, and thus represents the trivial class in the relative second homology of the cobordism. 

Note that a cobordism $(W, \nu)$ with incoming end $(Y, \lambda \cup \lambda')$ naturally induces a cobordism with incoming end $(Y, \lambda)$ by precomposing with the tautological cobordism $(Y, \lambda) \to (Y, \lambda \cup \lambda')$. 
We will therefore identify the cobordism map induced by $(W, \nu)$ with the cobordism map induced by $(W, C_{\lambda} \cup S_{\lambda'} \cup \nu)$. Similar remarks apply to postcomposing with the reverse of the tautological cobordism. 
\end{remark}

We may finally explain the surface decorations $\Sigma_a, \Sigma_b$ appearing in Theorem \ref{exact}: 
\begin{enumerate}
\item View $W_a$ as the trace $W_n(K)$ of the $n$-framed surgery along $K$. Then $\Sigma_a$ is the standard 2-handle decoration associated to $W_a$.
\item View $W_b$ as obtained via $(-1)$-framed handle attachment along $J$, where $J$ is the image in $Y_n(K)$ of a \textit{negatively}-oriented meridian of $K$. Then $\Sigma_b$ is the standard 2-handle decoration associated to $W_b$. Note that this yields a map
\[
\Is(W_b, \Sigma_b) : \Is(Y_n(K), \lambda \cup \mu_K) \rightarrow \Is(Y_{n+1}(K), \lambda \cup \mu_K \cup \mu_J),
\]
where $\mu_K$ and $\mu_J$ denote meridians of $K$ and $J$, respectively. Since $\mu_K \cup \mu_J$ is nullhomologous in (the complement of $\lambda$ in) $Y_{n+1}(K)$, we may use the  isomorphism of Remark \ref{rem:tautological} to identify
\[
\Is(Y_{n+1}(K), \lambda \cup \mu_K \cup \mu_J) \cong \Is(Y_{n+1}(K), \lambda)
\]
and obtain a cobordism map 
\[
\Is(W_b, \Sigma_b) : \Is(Y_n(K), \lambda \cup \mu_K) \rightarrow \Is(Y_{n+1}(K), \lambda)
\]
by postcomposing  with the reverse of the tautological cobordism map; see Figure~\ref{fig:exact}. To explicitly describe $\Sigma_b$ with respect to this identification,  view the core of the attaching 2-handle as having boundary $- \mu_K$. Then $\Sigma_b$ is just the surface obtained from the cylinder $C_{\lambda \cup \mu_K}$ by capping off $\mu_K$ with the core of the  2-handle.
\end{enumerate}

\begin{remark} $W_a \cup W_b$ is diffeomorphic to the blow-up of $(n+1)$-framed handle attachment along $K$. Under this identification, $\Sigma_a \cup \Sigma_b$ is the union of the standard 2-handle decoration for  the $(n+1)$-framed handle attachment with a single copy of the exceptional divisor.
\end{remark}

\begin{figure}[t]
\center
\includegraphics[scale=0.8]{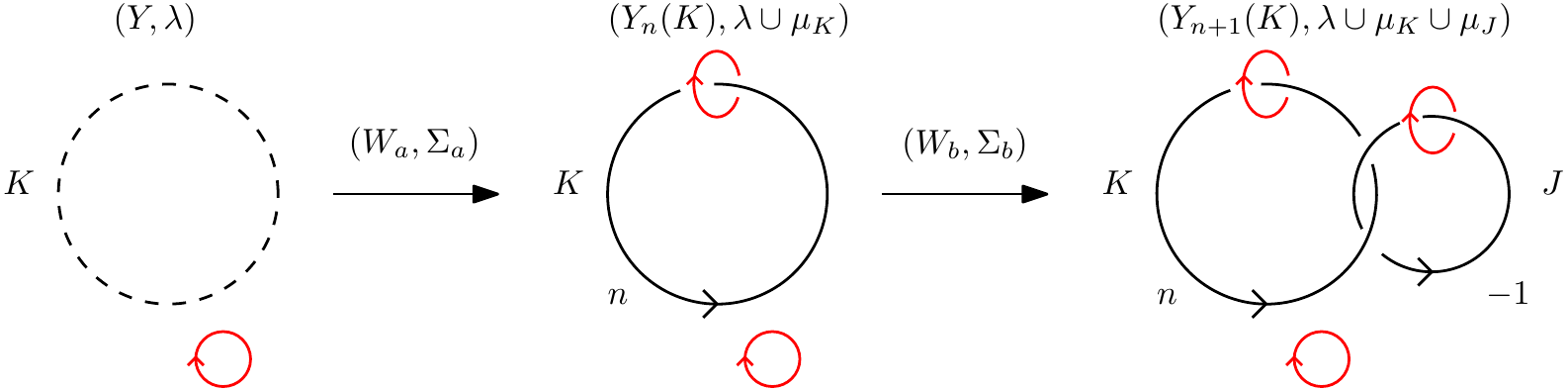}
\caption{Multicurve and surface decorations in the surgery exact triangle. Note that $J$ is oriented so that it has linking number $-1$ with $K$. Keeping careful track of the orientations, we see that $\mu_K \cup \mu_J$ is nullhomologous in $Y_{n+1}(K)-\lambda$.}\label{fig:exact}
\end{figure}

\subsection{Plumbings}\label{sec:2.4}
Let $\Gamma$ be a \emph{plumbing forest}. Recall that this is an undirected graph consisting of a disjoint union of trees, equipped with a function \[m: \Gamma \to \Z\] assigning an integer (framing) to each vertex of $\Gamma$. Associated to $\Gamma$, we have:
\begin{enumerate}
\item The plumbed 4-manifold $W_\Gamma$. This  is the simply connected 4-manifold  formed by plumbing disk bundles over the sphere as prescribed by $\Gamma$. The second homology $H_2(W_\Gamma)$ is an integer lattice  with a preferred basis; the preferred basis vectors are represented by embedded 2-spheres (the base spheres of these disk bundles) in one-to-one correspondence with the vertices of $\Gamma$. We denote the boundary of this manifold by $Y_\Gamma = \partial W_\Gamma$.  We will often remove a ball from the interior of $W_\Gamma$, and view it as a cobordism from $S^3$ to $Y_\Gamma$.
\item The intersection pairing $(-, -)_{\Gamma}$ on $H_2(W_\Gamma)\cong \Z^{|\Gamma|}$. This is given in the preferred basis of $H_2(W_\Gamma)$ by: 
\[
(v, w)_{\Gamma} = 
\begin{cases}
m(v) & \text{if } v = w, \\
-1 & \text{if $v$ and $w$ are adjacent in $\Gamma$}, \\
0 & \text{otherwise}.
\end{cases}
\]
Note that our convention in the case where $v$ and $w$ are adjacent differs from the standard one by a sign. We can think of this convention as being part of the data of $\Gamma$, in which edges of $\Gamma$ are labeled by $-1$ rather than the more standard $+1$. We will have occasion to consider the standard convention in Section \ref{sec:7.1}.
We refer to $\Gamma$ as \textit{negative-definite} if the pairing $(-, -)_\Gamma$ is negative-definite. Recall that if $v_1,\dots,v_n$ are the vertices of $\Gamma$, then the matrix \[A_\Gamma=((v_i,v_j)_\Gamma)\]  gives a presentation for $H_1(Y_\Gamma)$. In particular, if $\det A_\Gamma\neq 0$, as is the case when $\Gamma$ is negative-definite, then $Y_\Gamma$ is a rational homology sphere.
\item The link surgery diagram $\mathcal{L}_\Gamma$. This is  the framed link formed from oriented, framed unknots according to $\Gamma$, with linking matrix $A_\Gamma$. Attaching 2-handles to $B^4$ along the components of this link yields $W_\Gamma$. For any vertex $v \in \Gamma$, we use $\mu_v$ to denote the image in $Y_\Gamma$ of the positively-oriented meridian of the corresponding unknot.
\end{enumerate}
We will tacitly identify a vertex $v \in \Gamma$ with either:  the corresponding basis vector in $H_2(W_\Gamma)$; the corresponding embedded 2-sphere in $W_\Gamma$; or the corresponding unknot in $\mathcal{L}_\Gamma$, depending on the context. For further generalities on plumbing graphs, see \cite{Neumann}.

\begin{remark}\label{rem:signconventions}
Our non-standard convention for the linking number of adjacent vertices in $\Gamma$ will be helpful for making the orientations of meridional curves work out nicely, as in Figure~\ref{fig:exact}. While not strictly necessary, we adopt this in order to reduce bookkeeping throughout the paper. 
\end{remark}

Denote by $\Char(\Gamma)$ the set of characteristic vectors associated to the intersection form of $W_\Gamma$. Recall that these are the elements of \[H^2(W_\Gamma)\cong \text{Hom}_\Z(H_2(W_\Gamma), \Z)\] (we will  use  these two groups interchangeably) represented by maps \[k: H_2(W_\Gamma) \rightarrow \Z\] such that 
\[
\ev{k}{x} \equiv x^2 \text{ mod } 2
\]
for all $x \in H_2(W_\Gamma)$. The lattice $H_2(W_\Gamma)$ acts on the set $\Char(\Gamma)$, with the action of $x \in H_2(W_\Gamma)$ sending $k$ to $k + 2x^*$. Here and elsewhere, $x^*$ denotes the Poincar\'e dual of $x$; this is the element of $\text{Hom}_\Z(H_2(W_\Gamma), \Z)$ characterized by the equality \[\ev{x^*}{y} = (x, y)_\Gamma\] for all $y \in H_2(W_\Gamma)$. We denote the orbit of $k$ under this action by $[k]$. 

\begin{remark}
\label{rmk:spinc}
 The elements of $\Char(\Gamma)$ are in one-to-one correspondence with $\spinc$ structures on $W_\Gamma$ via the map sending a $\spinc$ structure to its first Chern class. The  orbits of the action of $H_2(W_\Gamma)$  on $\Char(\Gamma)$ are  similarly in one-to-one correspondence with $\spinc$ structures on the boundary $Y_\Gamma$.
\end{remark}

In what follows, we will be interested in some special classes of plumbings. 
\begin{definition}\cite[Definition 1.1]{OSplumbed} A vertex $v$ of a plumbing forest $\Gamma$ is called \emph{bad} if \begin{equation}\label{bad}
{-}m(v) < d(v),
\end{equation}
where $d(v)$ is the degree (i.e.\ valency) of $v$. We call  $\Gamma$  an \emph{$n$-bad-vertex plumbing} if it has exactly $n$ bad vertices.
\end{definition}

Boundaries of negative-definite plumbings with at most one bad vertex were considered in the Heegaard Floer context by Oszv\'ath and Szab\'o in \cite{OSplumbed}. Note that if $Y$ is a Seifert fibered rational homology sphere with base orbifold $S^2$, then either $Y$ or $-Y$ bounds a negative-definite $1$-bad-vertex plumbing, given by the usual star-shaped plumbing graph. 

Given a plumbing forest $\Gamma$, we define the \textit{canonical class} \[K\in H^2(W_\Gamma)\cong \Hom_\Z(H_2(W_\Gamma),\Z)\] to be the class satisfying  \[\ev{K}{v} = - m(v) - 2\] for all vertices $v$. Note that if we view $W_\Gamma$ as the resolution of a normal surface singularity then $K$ is the first Chern class of the canonical bundle of $W_\Gamma$. The following definition is inspired by the algebraic geometry of such singularities, and first appeared in \cite{NemethiOS}:

\begin{definition}[{\cite[Definition 8.1]{NemethiOS}}] A negative-definite plumbing forest $\Gamma$ is  \textit{rational} if
\begin{equation*}
-(\ev{K}{x} + x^2)/2 \geq 1
\end{equation*}
 for all $x > 0$ in $H_2(W_\Gamma)$. Here, $x > 0$ means that $x\neq 0$ and the coefficients of $x$ are  nonnegative with respect to the preferred basis of $H_2(W_\Gamma)$. If $\Gamma$ is not rational, but there exists a vertex $v$ of $\Gamma$ such that decreasing the framing of $v$ yields a rational plumbing, then $\Gamma$ is said to be  \textit{almost-rational}.
\end{definition}

Rational plumbings arise naturally in our context via a result of N\'emethi, which characterises rational plumbings as precisely those with minimal lattice homology (see \cite[Section 6.2]{NemethiOS}). Every negative-definite $0$-bad-vertex plumbing is rational, while every negative-definite $1$-bad-vertex plumbing is almost-rational. The converse  is not true: the $E_8$ plumbing is rational, but  has one bad vertex.


\section{Lattice homology}\label{sec:3}

\subsection{Construction}\label{sec:3.1}
In this section, we define the version of lattice homology used in this paper. Our presentation differs slightly from those in \cite{OSplumbed} and \cite{Nemethi} in order to accomodate the multicurve and surface decorations required in the instanton setting. We explain how our version of lattice homology is related to the version used in  Heegaard Floer homology in Section \ref{sec:7.1}.

\begin{definition}[Lattice homology]\label{def:3.1}
Let $\Gamma$ be a negative-definite plumbing forest. Let $V_\Gamma$ be the  vector space consisting of all  finite linear combinations of the form $\sum_i z_i \otimes k_i$, with $z_i\in \C$ and $k_i\in \Char(\Gamma)$.
Define the \textit{lattice homology} $\Hl(\Gamma)$ of $\Gamma$ to be the quotient $\Hl(\Gamma) = V_\Gamma/{\sim}$, where $\sim$ denotes the equivalence relation generated by the following elementary relations:
\begin{enumerate}[label=(\Roman*)]
\item If $k\in \Char(\Gamma)$ is such that $|\ev{k}{v}| > - v^2$ for some vertex $v\in \Gamma$, set $1 \otimes k \sim 0$.
\item If $k\in \Char(\Gamma)$ is such that $\ev{k}{v} = \pm v^2$ for some vertex $v\in \Gamma$, set 
\[
1 \otimes k \sim (-1)^{v^2} \otimes (k \mp 2v^*).
\] 
Here,  $v^*$ refers to  the Poincar\' e dual of the homology class represented by $v$ in $W_\Gamma$. 
\end{enumerate}
\end{definition}

\begin{remark} Note that $\Hl(\Gamma)$ is a finite-dimensional quotient of an infinite-dimensional complex vector space. Indeed, if $1\otimes k$ is non-zero in $\Hl(\Gamma)$, then necessarily \[v^2 \leq \ev{k}{v} \leq -v^2\] for all $v\in \Gamma$, and there are finitely many such elements $k$. A characteristic element $k$ for which $\ev{k}{v} = \pm v^2$ for some vertex $v$ is called \textit{extremal}. 
\end{remark}

We will sometimes abuse notation and write $k$ instead of  $1 \otimes k$. Note that a relation of Type II relates one extremal characteristic element to another extremal element, but the latter may be zero due to a relation of Type I applied to a vertex adjacent to $v$; see Example~\ref{ex:3.4}.

\begin{example}\label{ex:3.3}
Let $\Gamma$ be the graph consisting of a single vertex $v$ with  framing $-p<0$, so that $Y_{\Gamma} = L(p, 1)$ and \[W_{\Gamma}:S^3 \to L(p,1)\] is the cobordism obtained by removing a ball from the disk bundle over the sphere with Euler number $-p$. We will refer to such a $W_{\Gamma}$ as a \textit{1-vertex cobordism}, and denote it by $W_p$. Let $k_i$ be the element of $\Char(\Gamma)$ which evaluates to $i$ on $v$; in particular,  $i \equiv p$ mod $2$. Then
\[
\Hl(\Gamma) = \text{Span}_{\C} \{k_{-p}, k_{-p + 2}, \cdots, k_{p - 2}, k_p\}.
\]
These generators are linearly independent, with the exception of the  relation $k_{-p} = (-1)^{p} k_{p}$. Hence $\Hl(\Gamma) \cong \C^p$.
\end{example}

\begin{example}\label{ex:3.4}
Let $\Gamma$ be the standard negative-definite $E_8$-plumbing, so that $Y_{\Gamma} = \Sigma(2, 3, 5)$. In Figure~\ref{fig:3.1}, we have displayed two characteristic elements on $\Gamma$. These are equivalent to each other via a relation of Type II applied to the central vertex; the characteristic element on the right is then equivalent to zero by a relation of Type I. 
It can be shown  in this example that the \textit{only} non-zero element of $\Hl(\Gamma)$ is the equivalence class of the characteristic element which takes the value zero on each vertex; see for example \cite[Section 3.2]{OSplumbed}. Hence $\Hl(\Gamma) \cong \C$.\end{example}

\begin{figure}[t]
\center
\includegraphics[scale=0.8]{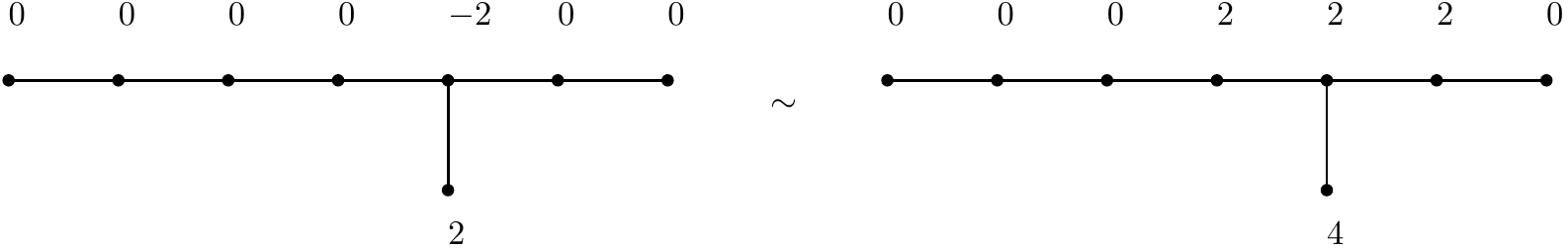}
\caption{Each characteristic element is described by labeling its value on the vertices of $\Gamma$; for example, the left-hand characteristic element is zero on all but two vertices of $\Gamma$. Note that adjacent vertices have pairing $-1$, rather than $+1$.}\label{fig:3.1}
\end{figure}

\subsection{The map from lattice homology to framed instanton homology}\label{sec:3.2}
We now define a map from the lattice homology of a negative-definite plumbing forest $\Gamma$ to the framed instanton homology of the associated 3-manifold $Y_{\Gamma}$. 

\begin{definition}[Standard multicurve]\label{def:3.5}
Let $\Gamma$ be a negative-definite plumbing forest. We define the \textit{standard multicurve} $\lambda_{\Gamma}$ in $Y_{\Gamma}$ by
\[
\lambda_\Gamma = \sum_{\text{vertices } v} |m(v)| \cdot \mu_v
\]
where $\mu_v$ denotes the (positively-oriented) meridian of $v$ in the surgery diagram $\mathcal{L}_\Gamma$, and $|m(v)|\cdot\mu_v$ is to be interpreted as $|m(v)|$ parallel copies of $\mu_v$, as shown  in Figure~\ref{fig:3.2}.
\end{definition}

\begin{lemma}\label{lem:3.6}
We have that $[\lambda_{\Gamma}] = 0$  as a class in $H_1(Y_{\Gamma}, \Z/2\Z)$. In particular, there is an isomorphism
$\Is(Y_{\Gamma},\lambda_{\Gamma})\cong \Is(Y_\Gamma)$.
\end{lemma}
\begin{proof}
This follows immediately from the existence of a characteristic vector. More precisely, let $w = c_1v_1 + \cdots + c_nv_n$ be a vector in $H_2(W_\Gamma)$ such that $(w, x) \equiv (x, x)$ mod $2$ for all $x \in H_2(W_\Gamma)$, where $v_1,\dots,v_n$ are the vertices of $\Gamma$. Such a vector always exists---the set of such vectors fall into equivalence classes parameterised by $\spinc$ structures on $Y_{\Gamma}$, per Remark \ref{rmk:spinc}. Let $T_i$ be the glued-in solid torus corresponding to surgery on $v_i$, and denote the meridional disk in $T_i$ by $D_i$. Set
$S = \sum_i c_i D_i$. We claim that  the boundary of $S$ is a multicurve which coincides with $\lambda_{\Gamma}$ in mod $2$ homology, which proves the lemma. 

For this claim, let $a_{ij} = (v_i, v_j)$. The fact that $w$ is a characteristic vector, as applied to the evaluation $(w,v_i)$ and using the fact that $a_{ji}=a_{ij}$, implies that \begin{equation}\label{eq:ci}c_1a_{i1}  + \cdots + c_na_{in}\equiv a_{ii} \textrm{ mod }2\end{equation} for  $i=1,\dots, n$. On the other hand, the $j$th column vector of the matrix $(a_{ij})$ is precisely the relation in $H_1(Y_\Gamma)$ given by $\partial D_j$, in the generators $\mu_{v_1},\dots,\mu_{v_n}$. That is, \[\partial D_j = a_{1j}\cdot\mu_{v_1} + \dots + a_{nj}\cdot\mu_{v_n}.\] So, $\partial S$ is given by \[\sum_jc_j\partial D_j = \sum_j c_j(a_{1j}\cdot\mu_{v_1} + \dots + a_{nj}\cdot\mu_{v_n})=\sum_i (c_1a_{i1}  + \cdots + c_na_{in})\cdot \mu_{v_i},\] and the latter is equivalent mod 2 to \[\sum_i a_{ii}\cdot \mu_{v_i} \equiv  \lambda_\Gamma\textrm{ mod } 2,\] by \eqref{eq:ci}. This proves the claim.
\end{proof}

\begin{definition}[Standard disk system]\label{def:3.7}
Let $\Gamma$ be a negative-definite plumbing forest. Each component of $\lambda_\Gamma$ bounds a properly embedded disk in $W_\Gamma$, obtained by pushing the interior of the  obvious meridional disk in the surgery diagram $\mathcal{L}_\Gamma$ into the interior of $W_\Gamma$. We define the \textit{standard disk system $D_\Gamma$ in $W_{\Gamma}$} to be the union of these disks.\end{definition}

\begin{figure}[t]
\center
\includegraphics[scale=1]{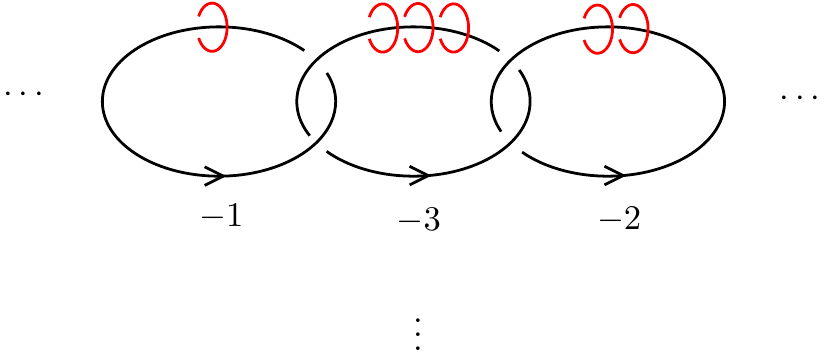}
\caption{A schematic link surgery diagram for $Y_{\Gamma} = \partial W_{\Gamma}$. The standard multicurve $\lambda_{\Gamma}$ consists of the collection of red meridians. Each of these is oriented so that it has linking number $+1$ with the relevant unknot. The standard disk system $D_\Gamma$ is obtained by pushing the obvious Seifert disks for these meridians  into the interior of $W_{\Gamma}$.}\label{fig:3.2}
\end{figure}

Let $\Gamma$ be a negative-definite plumbing forest. Let us fix, once and for all, a non-zero element \begin{equation}\label{eq:gen}x_0\in \Is(S^3, \emptyset)\cong \C.\end{equation} We define the vector space homomorphism \begin{equation}\label{eq:mapT}\Tsb : V_\Gamma \rightarrow \Is(Y_\Gamma, \lambda_\Gamma)\end{equation} by sending a generator $1\otimes k$ to
\[ 
\Tsb(1\otimes k) = \Is(W_\Gamma, D_\Gamma; k)(x_0)
\]
and extending $\C$-linearly. Here, \[\Is(W_\Gamma, D_\Gamma; k) : \Is(S^3, \emptyset) \rightarrow \Is(Y_\Gamma, \lambda_\Gamma)\] is the cobordism map associated to $W_\Gamma$, taken with the surface decoration $D_\Gamma$ and the characteristic element $k$.

\begin{remark}\label{rem:3.8}
Note that according to Lemma~\ref{lem:3.6}, there is an isomorphism 
\[
\Is(Y_\Gamma, \lambda_\Gamma) \cong \Is(Y_\Gamma, \emptyset) = \Is(Y_\Gamma).
\]
We may therefore view $\Tsb$ as a map with codomain $\Is(Y_\Gamma)$. The reader may then wonder why we have gone through the trouble of defining the multicurve and surface decorations $\lambda_\Gamma$ and $D_\Gamma$, rather than  using empty decorations everywhere. The reason, as discussed in Section~\ref{sec:2.3}, is that the cobordism maps in the surgery exact triangle in framed instanton homology require  decorations which may be nontrivial. A key element in the proof of Theorem~\ref{thm:1.1} is  to establish a corresponding surgery exact sequence in lattice homology which is natural with respect to the map  from lattice homology induced by $\Tsb$ (see below). As we shall see (for instance, in the proof of Lemma~\ref{lem:5.3}), this forces us to use nontrivial decorations in the definition of $\Tsb$.
\end{remark}

We  claim that $\Tsb$ respects the Type I and II relations of Definition \ref{def:3.1}, and therefore descends to a map from lattice homology:

\begin{proposition}\label{prop:3.9}
The map $\Tsb$ in \eqref{eq:mapT} descends to a well-defined complex-linear map
\[ 
\Ts : \Hl(\Gamma) \rightarrow \Iseven(Y_\Gamma, \lambda_\Gamma) \subset \Is(Y_\Gamma). 
\]
\end{proposition}

Proposition~\ref{prop:3.9} follows from a special case of the adjunction relations of Theorem \ref{thm:1.5}. The proof of these relations relies on a computation of  the  maps associated to the 1-vertex cobordisms  in Example \ref{ex:3.3},  carried out in Section~\ref{sec:4} below. We therefore postpone the proof of Proposition \ref{prop:3.9} until  the end of the next section.

\begin{remark} Although $\Hl(\Gamma)$ is defined using the specific plumbing forest $\Gamma$, the isomorphism class of $\Hl(\Gamma)$ is an invariant of $Y_\Gamma$. A proof of this appears in several places in the literature; see, for example, \cite[Proposition 3.4.2]{Nemethi}. 
In this paper, we will  use the following: let $\Gamma'$ and $\Gamma_{+1}$ be two plumbing forests such that $\Gamma_{+1}$ is obtained from $\Gamma'$ by blowing down a leaf with framing $-1$.\footnote{This choice of notation will become clear presently; see Appendix~\ref{sec:appendix}.} There is  a diffeomorphism $Y_{\Gamma'} \cong Y_{\Gamma_{+1}}$ which identifies the homology classes of   $\lambda_{\Gamma'}$ and $\lambda_{\Gamma_{+1}}$. We claim, moreover, that there is an isomorphism from $\Hl(\Gamma')$ to $\Hl(\Gamma_{+1})$ which makes the square
\[
\xymatrix@C=3.5em{
& \Is(Y_{\Gamma'}, \lambda_{\Gamma'}) \ar[r]^{\cong} & \Is(Y_{\Gamma_{+1}}, \lambda_{\Gamma_{+1}}) \\
& \Hl(\Gamma') \ar[r]^{\cong} \ar[u]^{\Ts} & \Hl(\Gamma_{+1}) \ar[u]^-{\Ts}
}
\]
commute. Here, the top row is the isomorphism afforded by the diffeomorphism $Y_{\Gamma'} \cong Y_{\Gamma_{+1}}$ and Remark~\ref{rem:tautological}. We prove this claim in Lemma \ref{lem:A.2} of Appendix~\ref{sec:appendix} (see also \cite[Proposition 2.5]{OSplumbed}), and   use it in Section \ref{sec:6} to prove  various isomorphisms relating lattice homology and framed instanton homology, toward the proof of Theorem \ref{thm:1.1}.
\end{remark}


\section{The adjunction relations}\label{sec:4}
We now turn to a careful analysis of the maps induced by the 1-vertex cobordisms  in Example \ref{ex:3.3}, and use this  to prove the adjunction relations of Theorem \ref{thm:1.5}. Along the way, we    prove Proposition \ref{prop:3.9},  which is a special case of these relations, and  asserts  that the map $\Tsb$ in \eqref{eq:mapT} descends to a map \[\Ts:\Hl(\Gamma)\to \Iseven(Y_\Gamma,\lambda_\Gamma).\] In fact, we  prove more, namely that $\Ts$ is an isomorphism for the graphs  associated to 1-vertex cobordisms; see Proposition \ref{prop:4.4} below.

Let us  establish some notation. Let $p>0$ be an integer, so that $L(p, 1)$ is $(-p)$-surgery on an unknot $v$. There is  a surgery triad given by the sequence of manifolds $S^3, L(p, 1), L(p-1, 1)$, associated to the $\infty, (-p), (1-p)$-surgeries on $v$. We define the following:

\begin{enumerate}
\item As in Example \ref{ex:3.3}, let \[W_p:S^3 \to L(p,1)\] be the 1-vertex cobordism represented by a graph with a single  vertex $v$ with framing $-p$. Let  $k_i$ be the characteristic element in $\text{Hom}_\Z(H_2(W_p), \Z)$ with 
\[
\ev{k_i}{v} = i
\]
(so that $i \equiv p$ mod $2$). We equip $L(p, 1)$ with the standard multicurve of Definition~\ref{def:3.5}, which we denote by $\lambda_p$, and we equip $W_p$ with the  standard disk system of Definition~\ref{def:3.7}, which we denote by $D_p$, associated to the  graph with vertex $v$.

\item Let \[W_{p, p-1}:L(p,1)\to L(p-1,1)\] be the cobordism obtained by attaching a  $(-1)$-framed 2-handle  along a (negatively-oriented) meridian $x$ of $v$ in $L(p,1)$. Let \[\smash{W_p \cup W_{p, p-1}}:S^3 \to L(p-1,1)\] be the composition of $W_p$ with $W_{p,p-1}$. Its second homology is generated by $v$ and $x$ (we are also using $v$ and $x$ here to denote the  $2$-spheres obtained by capping off the disks in $B^4$ bounded by the \emph{unknots} $v$ and $x$ with the corresponding $2$-handle cores). From this perspective, $H_2(W_{p, p-1}) \cong \Z$ is generated by the sphere \[x'=-px+v\] of square $-p(p-1)$, as $x'$ is the primitive integer linear combination of $x$ and $v$ which does not intersect $v$; see  Figure~\ref{fig:4.1}. Let $s_j$ be the characteristic element in $\text{Hom}_\Z(H_2(W_{p,p-1}), \Z)$ with 
\[
\ev{s_j}{x'} = j.
\]
Let $D_{p, p-1}$ be the properly embedded surface in $W_{p, p-1}$ obtained from the cyclinder $C_{\lambda_p}$ over $\lambda_p$ by capping off one of its $p$ outgoing boundary components (a copy of $\mu_v$ which we can identify with $x$) with the core of the $2$-handle attached along $x$.  The incoming boundary of $D_{p, p-1}$ is   the standard multicurve $\lambda_p$ in $L(p,1)$, and its outgoing boundary is the standard multicurve $\lambda_{p-1}$ in $L(p-1, 1)$.

\item The composition of $W_p$ with $W_{p,p-1}$ is diffeomorphic to the blow-up of $W_{p-1}$:
\[
\smash{W_p \cup W_{p, p-1} \cong W_{p-1} \# \overline{\CPr}^2}. 
\]
As a blow-up, the most natural basis for the second homology is given by the pair $\{v', x\}$, where $v'=v-x$; see Figure~\ref{fig:4.1}. 
Let  $t_{a, b}$ denote the characteristic element for this blow-up defined by 
\[
\ev{t_{a, b}}{v'} = a \text{ and } \ev{t_{a, b}}{x} = b. 
\]
Since $x' = -px+v$ in the composite cobordism, we have \[\ev{t_{a, b}}{v} = a + b \text{ and } \ev{t_{a, b}}{x'} = a - (p-1)b,\] which implies that $t_{a, b}$ restricts to $k_{a+b}$ on $W_p$ and to $s_{a - (p-1)b}$ on $W_{p, p-1}$. The union  \[D_p\cup D_{p, p-1}\subset W_p\cup W_{p,p-1},\]  viewed as a properly embedded surface  in the blowup of $W_p$,   is easily seen to be homologous to the union of the usual disk system $D_{p-1}$ for $W_{p-1}$ with a  copy of the exceptional divisor $E$ (which is in fact the 2-sphere $x$).
\end{enumerate}

\begin{figure}[t]
\center
\includegraphics[scale=0.8]{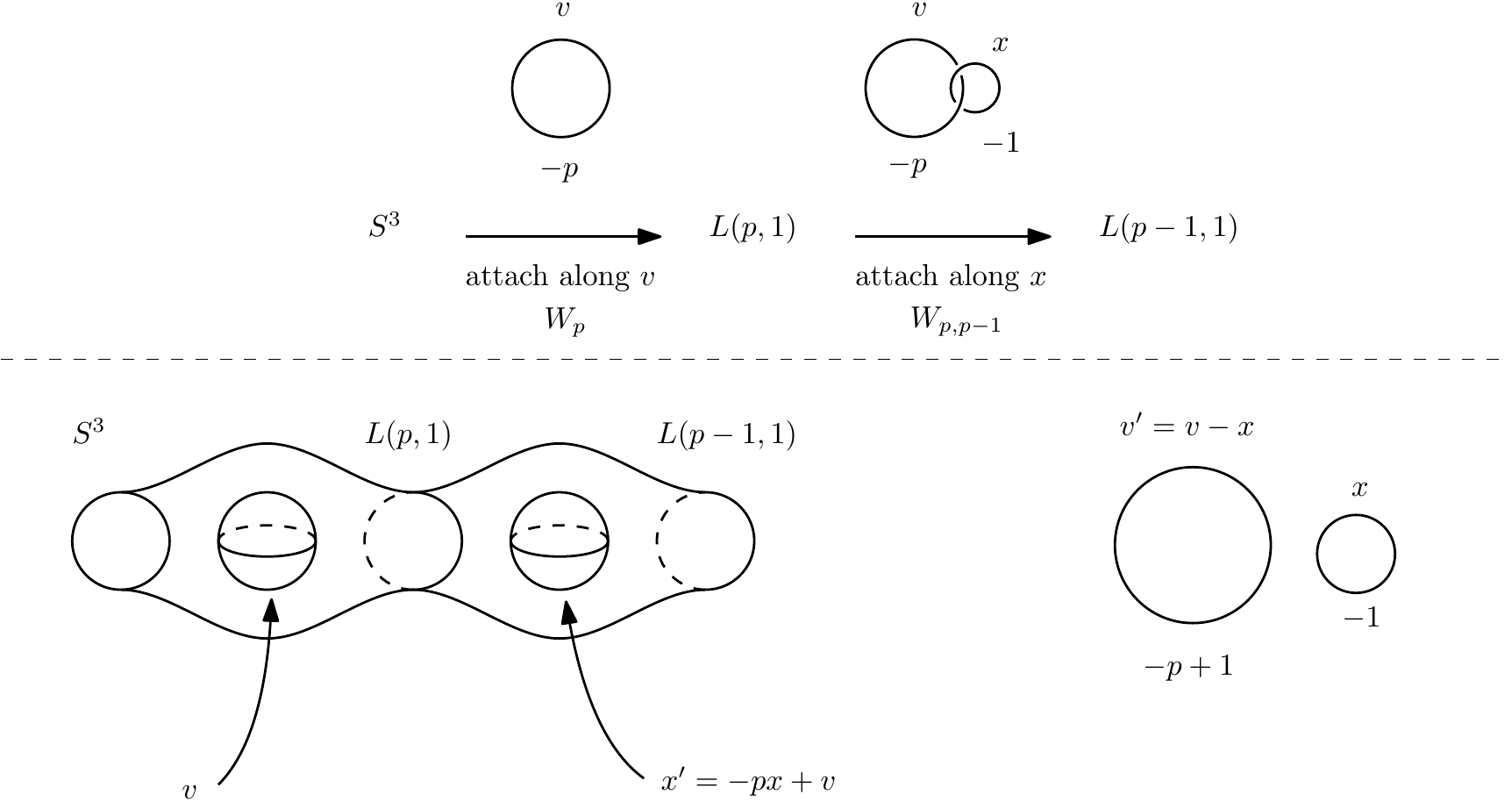}\caption{Top:  $W_p$ and $W_{p, p-1}$. Bottom: two views of $W_p \cup W_{p, p-1}\cong W_{p-1} \# \overline{\CPr}^2$. The bottom left shows the generators of the second homologies of $W_p$ and $W_{p, p-1}$. Note that $v$ and $x'$ do not intersect. 
The bottom right shows $W_p \cup W_{p, p-1}$ as a blow-up of $W_{p-1}$.}\label{fig:4.1}
\end{figure}

\begin{remark} Let $\lambda \subset S^3$ be the multicurve consisting of $p-1$ copies of the meridian $\mu_v$. Then the surface $D_{p, p-1}$ is   precisely the surface $\Sigma_b$  in the notation of Theorem \ref{exact}, for the surgery triad $S^3,L(p,1),L(p-1,1)$ associated to surgeries on the unknot $v$. Likewise, $D_p$ is  the composition of the surface $\Sigma_a$ therein with the  surface in the  tautological cobordism $(S^3,\emptyset)\to (S^3, (p-1)\cdot \mu_v)$ obtained by capping these meridians with disks. Recall that the latter cobordism is  what is used to canonically identify  $\smash{\Is(S^3, \emptyset)}$ with $\smash{\Is(S^3, (p-1) \cdot \mu_v)}$, per Remark~\ref{rem:tautological}. We therefore have a surgery exact triangle \[
\cdots \rightarrow \Is(S^3, \emptyset) \xrightarrow{\Is(W_p, D_p)} \Is(L(p,1), \lambda_p) \xrightarrow{\Is(W_{p,p-1}, D_{p,p-1})} \Is(L(p-1,1), \lambda_{p-1}) \rightarrow \cdots
\]  by  Theorem \ref{exact} and Remark~\ref{rem:tautological}.
\end{remark}

Lemmas \ref{lem:4.1} and \ref{lem:4.3} below are the keys  for  both  Theorem \ref{thm:1.1} and Proposition \ref{prop:3.9}. In particular, these two lemmas will be used to show that the map \[\Tsb : V_\Gamma \rightarrow \Is(Y, \lambda_\Gamma)\] in \eqref{eq:mapT} respects the Type I and  II relations of Definition \ref{def:3.1}, respectively, and thus descends to a map \[\Ts : \Hl (\Gamma) \rightarrow \Is(Y, \lambda_\Gamma).\] In both lemmas, we will use the following notation. For any integer $p > 0$ and any integer $i \equiv p$ mod $2$, let
\[ 
x_{p, i} = \Tsb(1\otimes k_i) = \Is(W_p, D_p; k_i)(x_0) \in \Is(L(p, 1), \lambda_p),
\]
where $x_0$ is the fixed generator of $\Is(S^3, \emptyset)\cong \C$ from \eqref{eq:gen}.

\begin{lemma}\label{lem:4.1} If $|i| > p$, then $x_{p, i} = 0$.
\end{lemma}
\begin{proof}
Our argument is based on \cite[Section 5.3]{BSLspace}.  The key idea is that we cannot apply the adjunction inequality of Theorem~\ref{thm:decomposition} to the generator $v$ of $H_2(W_p)$, because it has negative self-intersection.  We get around this by composing with a cobordism $W$ that contains a class $y$ of positive self-intersection, and then applying the adjunction inequality to a surface in the class $v+ny$ for $n \gg 0$.  If $W$ induces an injection on Floer homology, then any vanishing results we can deduce for $v+ny$ will imply some sort of vanishing results for the original class $v$.

Let $T_{2,5}$ be a $(2,5)$-torus knot contained in a small ball inside $L(p, 1)$. Let \[(W,C_{\lambda_p}):(L(p, 1),\lambda_p) \to (L(p, 1) \# S^3_{2}(T_{2,5}),\lambda_p)\]  be the cobordism obtained by attaching a $(+2)$-framed 2-handle along $T_{2,5} \subset L(p,1)$, where $\smash{C_{\lambda_p}}$ is the cylinder over $\lambda_p$. The second homology of $W$ is generated by an  element $y$, represented by a genus-2 surface, with $y^2 = +2$. Consider the class $v + ny \in H_2(W_p \cup W)$, where $v$ is the usual generator of $W_p$ and $n$ is a large positive integer. This has self-intersection $2n^2 - p$, which is greater than zero if $n$ is large enough. Note that the class of $v + ny$ can be represented by an embedded surface $\Sigma$ with Euler characteristic $\chi(\Sigma)=-2n^2$. 

It follows from the proof of \cite[Proposition 5.13]{BSLspace} (see Remark~\ref{rem:4.2} below) that the  map \[\Is(W, C_{\lambda_p}):\Is(L(p, 1),\lambda_p) \to \Is(L(p, 1) \# S^3_{2}(T_{2,5}),\lambda_p)\] is injective. The adjunction inequality of Theorem \ref{thm:decomposition} implies that the summand $\Is(W, C_{\lambda_p}; u)$ of this map is non-zero only if  $\ev{u}{y} = 0$. Thus, \[\Is(W, C_{\lambda_p})=\Is(W, C_{\lambda_p}; u_0),\] where $u_0$ is the characteristic element with $\ev{u_0}{y}=0$. Now, given any characteristic element $k_i$ on $W_p$, consider the characteristic element $k_i \cup u_0$ on $W_p \cup W$ restricting to $k_i$ on $W_p$ and $u_0$ on $W$. Suppose the  map $\Is(W_p, D_p; k_i)$ is non-zero. Then the composition
\[
\Is(W, C_{\lambda_p}; u_0) \circ \Is(W_p, D_p; k_i) = \Is(W_p \cup W, D_p \cup C_{\lambda_p}; k_i \cup u_0)
\]
is non-zero by the injectivity of the first map. By another application of the adjunction inequality, applied to the surface representing $v + ny$, this can only occur as long as
\[
|i| + (2n^2 - p)= |\ev{k_i \cup u_0}{\Sigma}| + \Sigma \cdot \Sigma \leq -\chi(\Sigma)= 2n^2. 
\]
Hence $|i| \leq p$, and we are done.
\end{proof}

\begin{remark}\label{rem:4.2}
Proposition 5.13 of \cite{BSLspace} states that the map induced by the trace of $1$-surgery on $T_{2,5}\subset S^3$ is injective. However, it is shown in the course of the proof that $\smash{\dim \Is(S^3_{2}(T_{2,5})) = 4}$. Note also that $\smash{\dim \Is(S^3_{3}(T_{2,5})) = 3}$ by \cite[Lemma 5.11]{BSLspace}. The surgery exact triangle  associated to the  triad $S^3, \smash{S^3_2(T_{2,5})}, \smash{S^3_3(T_{2,5})}$ then shows that the map induced by the trace of 2-surgery on $T_{2, 5} \subset S^3$ is injective. The  case in which $T_{2, 5}$ is contained inside a small ball in  $L(p,1)$ then follows by the naturality of cobordism maps with respect to split cobordisms, as in the proof of \cite[Lemma 5.14]{BSLspace}.
\end{remark}

The second of our two key lemmas is the following.

\begin{lemma}\label{lem:4.3}
The elements
\[
\{x_{p, -p}, x_{p, -p +2}, \ldots, x_{p, p - 2}, x_{p, p}\}
\]
are linearly independent, except for the single relation $x_{p, -p} = (-1)^p x_{p, p}$.
\end{lemma}
\begin{proof}
We proceed by induction. Suppose $p=1$. The composition of the cobordism \[(W_1, D_1):(S^3,\emptyset) \to (S^3,\lambda_1)\] with a  tautological cobordism $(S^3,\lambda_1) \to (S^3,\emptyset)$ is obtained from the identity cobordism $S^3\times I$ with trivial decoration by first blowing up,  and then changing the trivial decoration to the decoration given (up to sign) by the exceptional class $\alpha = \pm E$. The blow-up formula and sign relation of Theorem \ref{thm:decomposition} then tell us that \[x_{1,-1} = -x_{1,1} \neq 0,\] proving the lemma in this case.

Next,  assume  that $p\geq 2$ and  the lemma holds for $p - 1.$ By the composition law of Theorem \ref{thm:decomposition} and our  discussion in point (3) at the beginning of this section, we have that
\begin{equation}\label{eqn:comp}
\Is(W_p \cup W_{p, p-1}, D_p \cup D_{p, p-1}; t_{a,b}) = \Is(W_{p, p-1}, D_{p, p-1}; s_{a - (p-1)b}) \circ \Is(W_p, D_p; k_{a+b}).
\end{equation}
Since $W_p \cup W_{p, p-1}\cong W_{p-1}\# \overline{\CPr}^2$ and $t_{a,b}$ evaluates to $b$ on the exceptional sphere of the blow-up,  the map is zero unless $b=\pm1$ by the blow-up formula of Theorem \ref{thm:decomposition}. Moreover,  we have
\begin{align*}
\Is(W_{p-1} \# \overline{\CPr}^2, D_p \cup D_{p, p-1}; t_{a,\pm 1}) &= (\pm 1) \cdot \Is(W_{p-1} \# \overline{\CPr}^2, D_{p-1}; t_{a,\pm 1}) \\
&= (\pm 1/2) \cdot \Is(W_{p-1}, D_{p-1}; k_a),
\end{align*}
where the initial signs come from the sign relation of Theorem \ref{thm:decomposition} together with the fact,  noted earlier earlier in  this section, that $D_p \cup D_{p, p-1}$ is homologous to the  union of $D_{p-1}$ with the exceptional divisor. By Lemma \ref{lem:4.1}, this  map vanishes unless $|a|\leq p-1$. We have therefore found that the composition in  \eqref{eqn:comp} vanishes unless $|a| \leq p -1$ and $b = \pm 1$. Considering all pairs $(a, b)$ satisfying these inequalities, we obtain the calculation for the maps \[\Is(W_p,D_p;k_i) \textrm{ and } \Is(W_{p,p-1},D_{p,p-1}; s_j)\] displayed in Figure~\ref{fig:4.2}.

\begin{figure}[t]
\center
\includegraphics[scale=0.9]{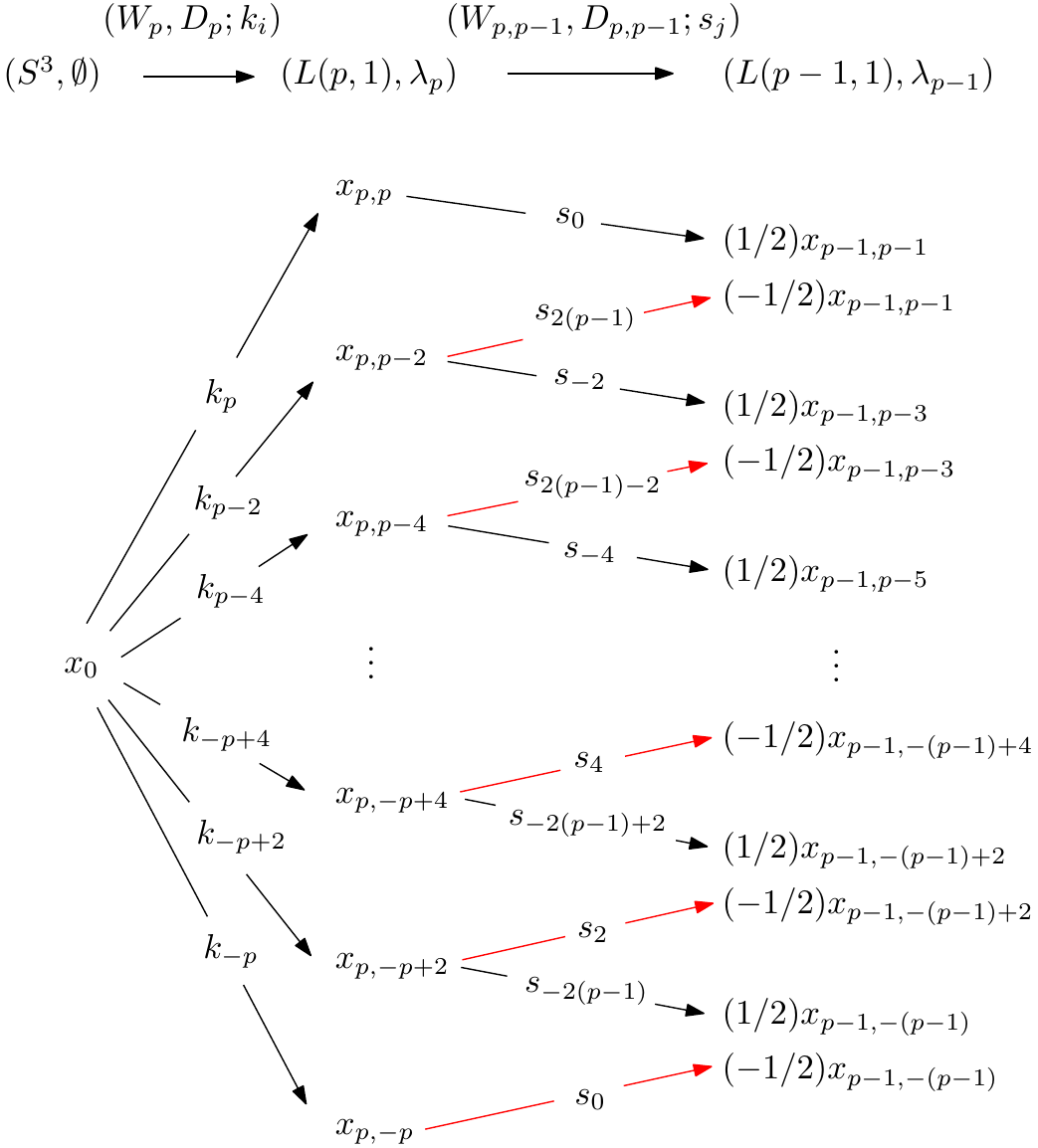}
\caption{Recall that $t_{a,b}$ restricts to $k_{a + b}$ on $W_p$ and $s_{a - (p-1)b}$ on $W_{p,p-1}$. Each pair $(a,b)$ gives one path from left to right in the above diagram. For example, $t_{p-1, +1}$ restricts to $k_p$ on $W_p$ and $s_0$ on $W_{p-1, p}$; the composition of these cobordism maps is represented by the uppermost path from left to right.}\label{fig:4.2}
\end{figure}

Given this calculation, it is not hard to see that the vectors $x_{p, p -2}, x_{p, p -4},\ldots ,x_{p, -p+2},x_{p,-p}$ form a basis for $\Is(L(p,1), \lambda_p) \cong \C^p$. Indeed, suppose
\[
a_{p -2}\cdot x_{p, p -2}+ a_{p -4} \cdot x_{p, p -4}+ \cdots +a_{-p+2}\cdot x_{p, -p+2}+ a_{-p}\cdot x_{p,-p} = 0.
\]
For  $i=0, \ldots, p-1$,  applying  the map $\Is(W_{p, p-1}, D_{p, p-1}; s_{2(p-1)-2i})$ to both sides of the above equation (see the red arrows in Figure~\ref{fig:4.2}) yields  \[(-1/2)\cdot a_{p-2i-2}\cdot x_{p-1, p -1-2i}=0.\] Since every $x_{p-1, p -1-2i}$ is nonzero by induction, this means  that $a_{p-2i-2}=0$ for all $i=0,.\dots,p-1$. We have thus shown that the $p$ vectors $x_{p, p -2}, x_{p, p -4},\ldots ,x_{p, -p+2},x_{p,-p}$ are linearly independent and therefore form a basis for $\Is(L(p,1), \lambda_p) \cong \C^p$ as claimed. 

It remains to prove that $x_{p,p}=(-1)^p\cdot x_{p,-p}$. For this, let us write 
\[
x_{p,p}= a_{p -2}\cdot x_{p, p -2}+ a_{p -4} \cdot x_{p, p -4}+\cdots +a_{-p+2}\cdot x_{p, -p+2}+ c\cdot x_{p,-p}.
\]
Applying $\Is(W_{p, p-1}, D_{p, p-1}; s_i)$ to both sides of this equation for each $i \neq 0$, and using our inductive assumption,  we may conclude that $x_{p,p}= c\cdot  x_{p,-p}$. Applying $\smash{\Is(W_{p, p-1}, D_{p, p-1}; s_0)}$ to both sides then yields $(1/2)\cdot x_{p-1,p-1}=c(-1/2)\cdot x_{p-1,-(p-1)}$. It follows from the inductive hypothesis that 
\[(1/2)\cdot x_{p-1,p-1}=c(-1/2)(-1)^{p-1}\cdot x_{p-1,p-1} \ .\] 
Hence $c=(-1)^p$, completing the proof. 
\end{proof}

We next prove that the map $\Ts$ of Proposition \ref{prop:3.9}  is well-defined and is an isomorphism for the 1-vertex plumbings in Example \ref{ex:3.3}:

\begin{proposition}\label{prop:4.4}
If $\Gamma$ is a graph with a single vertex with framing $-p<0$, then 
\[
\Ts : \Hl(\Gamma) \rightarrow \Is(L(p, 1), \lambda_p)
\]
is a well-defined isomorphism.
\end{proposition}
\begin{proof}
Lemmas~\ref{lem:4.1} and \ref{lem:4.3}  immediately imply that the map \[\Tsb : V_\Gamma \rightarrow \Is(L(p, 1), \lambda_p)\]  in \eqref{eq:mapT} respects the Type I and  II relations of Definition \ref{def:3.1} for this $\Gamma$. This shows that it indeed descends to a map \[\Ts : \Hl(\Gamma) \rightarrow \Is(L(p, 1), \lambda_p),\] proving the well-definedness claim. To show that the induced map $\Ts$ is an isomorphism, simply note that  both groups are isomorphic to $\C^p$, and $\Ts$ is injective by Lemma~\ref{lem:4.3}.
\end{proof}

We now prove Proposition \ref{prop:3.9}, which states that the map \[\Ts:\Hl(\Gamma)\to \Is(Y_\Gamma,\lambda_\Gamma)\] induced by $\Tsb$ is well-defined for \emph{any} negative-definite plumbing forest $\Gamma$, with image in the evenly-graded summand of framed instanton homology.

\begin{proof}[Proof of Proposition \ref{prop:3.9}]
Let $\Gamma$ be a negative-definite plumbing forest.  To prove that $\Tsb$ satisfies the Type I relation of Definition \ref{def:3.1},  let $v$ be a vertex of $\Gamma$ with framing $-p<0$. Suppose $k$ is a characteristic element on $\Gamma$ such that $|\ev{k}{v}| > p$, so that $1\otimes k \sim 0$ in $\Hl(\Gamma)$. We must show that
\[
\Is(W_\Gamma, D_\Gamma; k)(x_0) = 0.
\]
Taking the normal bundle of $v$ in $W_\Gamma$ gives an embedding of the 1-vertex cobordism $W_p$ in $W_{\Gamma}$. Indeed, we can view $W_{\Gamma}$ as the composition of $W_p$ and the closure  $W_p^c$ of its complement.
The standard disk system $D_\Gamma\subset W_\Gamma$  intersects $W_p$ in its standard disk system  $D_p$. Let us denote the intersection of $D_\Gamma$ with $W_p^c$ by $D_p^c$. By the composition law of Theorem \ref{thm:decomposition}, we then have
\[
\Is(W_\Gamma, D_\Gamma; k)(x_0) = \Is(W_p^c, D_p^c; k|_{W_p^c}) \circ \Is(W_p, D_p; k|_{W_p})(x_0).
\]
By Lemma~\ref{lem:4.1}, this is zero, since  $\smash{\left|\ev{k|_{W_p}}{v}\right| = |\ev{k}{v}| > p}$. Noting that $k$ and $k \pm 2v^*$ restrict to the same characteristic element on $\smash{W_p^c}$, a similar argument, but using Lemma \ref{lem:4.3}, proves that $\Tsb$ respects the  Type II relation as well. This proves the well-definedness claim. The claim that $\Ts$ maps into the evenly-graded summand \[\Iseven(Y_\Gamma, \lambda_\Gamma)\subset\Is(Y_\Gamma, \lambda_\Gamma)\] follows immediately from the grading shift formula in \eqref{eqn:degree} and the fact that $\Is(S^3,\emptyset)\cong \C$ is supported in even grading.
\end{proof}

 Theorem~\ref{thm:1.5} follows by almost exactly the same reasoning as the proofs of Lemmas \ref{lem:4.1}, \ref{lem:4.3}, and Proposition \ref{prop:3.9}; one simply has to take into account more general decorations, largely using the sign relation of Theorem \ref{thm:decomposition}.

\begin{proof}[Proof of Theorem~\ref{thm:1.5}]
Let us first suppose that $W = W_p$ is the 1-vertex cobordism of Example~\ref{ex:3.3}, with embedded 2-sphere $\Sigma$ satisfying \[\Sigma\cdot\Sigma=-p<0.\] As in    Remark \ref{rmk:dependence}, the map $\Is(W, \nu; s)$ depends (up to an overall  sign, independent of $s$) only on the class $[\nu] \in H_2(W, \partial W; \Z/2\Z).$
Assume  that $p$ is odd. Then \[H_2(W, \partial W; \Z/2\Z) \cong \Z/2\Z\]  is generated by the class of either the standard disk system $D_p$ or $\Sigma$ (which coincide). There are thus two possibilities for $[\nu]$. If $[\nu]\neq 0$, then $[\nu] = [D_p]$ and the fact that $\Is(W, \nu; s)$ satisfies the identities (1) and (2)  of Theorem \ref{thm:1.5} follows directly from  Lemmas \ref{lem:4.1} and \ref{lem:4.3}. Otherwise,  if $[\nu]=0$, then the class of $\nu$ is represented by  $D_p \cup \Sigma$. A quick  application of the sign relation in Theorem \ref{thm:decomposition}, comparing $\Is(W, \nu; s)$ with $\Is(W, \nu-\Sigma; s)$,  and the preceding case shows that $\Is(W, \nu; s)$ satisfies  identities (1) and (2) in this case as well.

Now assume that $p$ is even. Then \[H_2(W, \partial W; \Z/2\Z) \cong \Z/2\Z \oplus \Z/2\Z\]  is generated by the class of $\Sigma$ together with the class of the disk Poincar\'e dual to $\Sigma$. There are thus four possibilities for $[\nu]$. If $[\nu]=0$, then $\nu$ can be represented by $D_p$, since $D_p$ is an even number (namely, $p$) of copies of the dual disk to $\Sigma$. In this case, the fact that $\Is(W, \nu; s)$ satisfies the identities (1) and (2)   follows directly from  Lemmas \ref{lem:4.1} and \ref{lem:4.3}.
Theorem \ref{thm:1.5} in the case $[\nu] = [\Sigma]$ then follows from the preceding by an application of the sign relation of Theorem \ref{thm:decomposition}. Theorem \ref{thm:1.5} in the case where $[\nu]$ is  the class of the dual disk to $\Sigma$ does not follow directly from our work thus far. For this case, however, we merely repeat the arguments of this section, except that we alter $\lambda_p$ and $D_p$ throughout to have one fewer meridional curve/disk. The same arguments as in Lemmas \ref{lem:4.1} and \ref{lem:4.3}  apply, except that in Lemma \ref{lem:4.3}  we get the equality $x_{p, -p} = (-1)^{p-1} x_{p, p}$. The difference in sign comes from the difference in sign for  the base case $p = 1$, which follows from the sign relation of Theorem \ref{thm:decomposition}. This proves the theorem when $[\nu]$ is the class of the dual disk. Theorem \ref{thm:1.5} in the final case, where $[\nu]$ is the sum of the classes represented by $\Sigma$ and the dual disk, then follows from the preceeding by applying the sign relation. We have therefore proven Theorem \ref{thm:1.5} in the case $W=W_p$.

Now let $(W, \nu)$ be a  general cobordism from $(Y_1, \lambda_1)$ to $(Y_2, \lambda_2)$ as in the hypothesis of Theorem \ref{thm:1.5}. At least one of $Y_1$ or $Y_2$ is a rational homology sphere; let us assume first that $Y_1$ is. There is a copy of $L(p, 1)$ in $W$, given as the boundary of  a neighborhood of the $(-p)$-framed 2-sphere $\Sigma$.  Join $Y_1$ to this copy of $L(p,1)$ by an arc in $W$. A neighborhood of these two 3-manifolds and this arc is then simply the cobordism  \[W_1:Y_1\to Y_1 \# L(p, 1)\] given by the trace of $(-p)$-surgery on an unknot in $Y_1$. This decomposes $W$ into two pieces: $W_1$ and the closure of its complement, \[W_2:Y_1 \# L(p, 1)\to Y_2.\]
Note that $W_1$ is a split cobordism formed from the identity cobordism from $Y_1$ to $Y_1$ and  the 1-vertex $W_p$ from $S^3$ to $L(p, 1)$. By \cite[Section 7.7]{Scaduto}, the cobordism maps on framed instanton homology are natural with respect to split cobordisms. Hence, the fact that the identities (1) and (2) of Theorem \ref{thm:1.5} are satisfied for $W_p$, as established above, implies that they hold for $W_1$ and thus for the cobordism $W$, via the composition law of Theorem \ref{thm:decomposition}. 

We assumed above that  $Y_1$ is a rational homology sphere. If this is not the case, then $Y_2$ is a rational homology sphere, and we consider instead the cobordism \[(W^\dagger,\nu^\dagger):(-Y_2,-\lambda_2)\to (-Y_1,-\lambda_1)\] obtained by flipping $(W,\nu)$ upside-down. We then apply the  argument above to this cobordism. The decomposition of $\Is(W,\nu)$ is insensitive to flipping upside-down, except that the individual summands all get dualized, so this proves Theorem \ref{thm:1.5} in this case as well.
\end{proof}

\begin{remark}
The reason we require in Theorem \ref{thm:1.5} the hypothesis that at least one of $Y_1$ or $Y_2$ is a rational homology sphere is because in the proof we need to consider the decomposition of the cobordism map  $W_1:Y_1 \to Y_1\#L(p,1)$, for instance, along homomorphisms from $H_2(W_1;\Z)\to \Z$. Technically, this decomposition was only proven in \cite{BSLspace} for cobordisms with $b_1=0$, which holds for $W_1$ above if and only if $Y_1$ is a rational homology sphere. In the end, however, this hypothesis in Theorem \ref{thm:1.5} is probably not  necessary  since $W_1$ is a split cobordism formed from $Y_1\times I$ with the trace of $(-p)$-surgery on the unknot in $S^3$, and this trace cobordism has $b_1=0$; see, for example, \cite[Proposition 6.6]{BSLspace}.
\end{remark}

\section{The lattice surgery exact sequence}\label{sec:5}
In this section, we establish the \textit{lattice surgery exact sequence}, a lattice homology analogue of the surgery exact triangle for $\Is$. The results and arguments in Section~\ref{sec:5.1} are nearly identical to those in  \cite[Lemma 2.9]{OSplumbed} (see also \cite{Greenesurgery}). In Section~\ref{sec:5.2}, we  show that the map $\Ts$ of Proposition \ref{prop:3.9}, from lattice homology to framed instanton homology, commutes with the maps in these two exact sequences.

\subsection{The lattice surgery exact sequence}\label{sec:5.1}
Let $\Gamma$ be a negative-definite plumbing forest, and let $v$ be a vertex in $\Gamma$ with  decoration $-p<0$. We consider two modified plumbing forests $\Gamma - v$ and $\Gamma_{+1}$. The first is formed by deleting $v$ from $\Gamma$, while the second is obtained by increasing the framing of $v$ to $-p + 1$. Clearly, $\Gamma - v$ is still negative-definite; we will assume below that $\Gamma_{+1}$ is also negative-definite. It is easy to see that the manifolds $Y_{\Gamma - v}, Y_{\Gamma}$, and $Y_{\Gamma_{+1}}$ form a surgery triad. This motivates the construction of maps
\[
\Hl(\Gamma - v) \xrightarrow{\ \ \mathbb{A}\ \ } \Hl(\Gamma) \xrightarrow{\ \ \mathbb{B}\ \ } \Hl(\Gamma_{+1}),
\]
which we now define.

To define $\mathbb{A}$, think of $W_{\Gamma - v}$ as a submanifold of $W_{\Gamma}$. For any element $k \in \Char(\Gamma - v)$, let $\mathbb{A}(k)$ be the sum over all elements $k' \in \Char(\Gamma)$ that restrict to $k$ on $W_{\Gamma-v}$. Explicitly, these are given by $k_i'$, where
\[
\ev{k_i'}{w} = 
\begin{cases}
\ev{k}{w}, & \text{for } w \neq v \\
i, & \text{for } w = v
\end{cases}
\]
for all $i \equiv p$ mod $2$. Hence in fact
\[
\mathbb{A}(k) = \sum_{i \equiv p \text{ mod } 2} 1 \otimes k_i'.
\]
Note that only finitely many $k_i'$ are non-zero in $\Hl(\Gamma)$. Indeed, for the purposes of mapping into $\Hl(\Gamma)$, we may restrict to $|i| \leq p$.

To define the map $\mathbb{B}$, let $k$ be an element in $\Char(\Gamma)$. Define two elements $k^{\pm}$ of $\Char(\Gamma_{+1})$ by
\[
\ev{k^{\pm}}{w} = 
\begin{cases}
\ev{k}{w}, & \text{for } w \neq v \\
\ev{k}{w} \pm 1, & \text{for } w = v.
\end{cases}
\]
We then set
\[
\mathbb{B}(k) = (-1/2) \otimes k^+ + (1/2) \otimes k^- \ .
\]

\begin{lemma}\label{lem:5.1}
The maps $\mathbb{A}$ and $\mathbb{B}$ descend to well-defined maps
\[
\Hl(\Gamma - v) \xrightarrow{\ \ \mathbb{A}\ \ } \Hl(\Gamma) \xrightarrow{\ \ \mathbb{B}\ \ } \Hl(\Gamma_{+1}).
\]
\end{lemma}
\begin{proof}
We first show that $\mathbb{A}$ respects the relations of Definition~\ref{def:3.1}. Let $k \in \Char(\Gamma - v)$ and let $u$ be a fixed vertex in $\Gamma - v$ with decoration $-q$. It is straightforward to check that $\mathbb{A}$ preserves relations of Type I. Thus, suppose that $\ev{k}{u} = -q$. Write
\[
s = k - 2(u, -)_{\Gamma - v}.
\]
Then $\smash{k \sim (-1)^{q} \otimes s}$, so we must show that 
\[
\mathbb{A}(k) \sim (-1)^{q} \otimes \mathbb{A}(s). 
\]
Now, for each $k_i'$, we have a relation of Type II associated to $u$:
\begin{equation}\label{eq:5.1}
k_i' \sim (-1)^{q} \otimes (k_i' - 2(u, -)_\Gamma).
\end{equation}
On the other hand, by definition
\begin{equation}\label{eq:5.2}
\ev{s_i'}{w} =
\begin{cases}
\ev{s}{w}, & \text{for } w \neq v \\
i, & \text{for } w = v
\end{cases}
\ =
\begin{cases}
\ev{k}{w} - 2(u, w)_{\Gamma - v}, & \text{for } w \neq v \\
i, & \text{for } w = v.
\end{cases}
\end{equation}
If $u$ is not adjacent to $v$, then it is easily checked from \eqref{eq:5.2} that $k_i' - 2(u, -)_\Gamma$ and $s_i'$ are equal, simply by evaluating them on each vertex of $\Gamma$. However if $u$ is adjacent to $v$, then the two differ in their evaluation on $v$. In particular, in this situation we have
\[
k_i' - 2(u, -)_\Gamma = s'_{i+2}.
\]
Combining this with \eqref{eq:5.1}, in both cases we have the equality of sums
\[
\mathbb{A}(k) = \left( \sum_{i \equiv p \text{ mod } 2} k_i' \right) \sim \left( (-1)^{q} \otimes \sum_{i \equiv p \text{ mod } 2} s_i' \right) = (-1)^{q} \otimes \mathbb{A}(s),
\]
as desired. The case where $\ev{k}{u} = q$ is entirely analogous.

We now show that $\mathbb{B}$ respects the relations of Definition~\ref{def:3.1}. Relations of Type I are again straightforward. For a relation of Type II, suppose that $\ev{k}{u} = -q$ and write 
\[
s = k - 2(u, -)_{\Gamma}.
\]
Then $\smash{k \sim (-1)^{q} \otimes s}$, so we must show that
\[
\mathbb{B}(k) \sim (-1)^{q} \otimes \mathbb{B}(s). 
\]
We subdivide into several cases. First suppose that $u \neq v$. Then $\ev{k^\pm}{u} = \ev{k}{u}$ and the decoration of $u$ in $\Gamma_{+1}$ is still $-q$, so we have relations
\[
k^{\pm} \sim (-1)^q (k^\pm - 2(u, -)_{\Gamma_{+1}}).
\]
Now, by definition
\begin{equation}\label{eq:5.3}
\ev{s^{\pm}}{w}
=
\begin{cases}
\ev{s}{w}, & \text{for } w \neq v \\
\ev{s}{w} \pm 1 & \text{for } w = v
\end{cases}
\ =
\begin{cases}
\ev{k}{w} - 2(u, w)_{\Gamma}, & \text{for } w \neq v \\
\ev{k}{w} - 2(u, w)_{\Gamma} \pm 1 & \text{for } w = v.
\end{cases}
\end{equation}
As long as $u \neq v$, it is clear that $\smash{(u, -)_{\Gamma_{+1}} = (u, -)_{\Gamma}}$. In this case, the characteristic elements $k^\pm - 2(u, -)_{\Gamma_{+1}}$ and $s^{\pm}$ are equal. Hence $k^{\pm} \sim (-1)^q \otimes s^{\pm}$, which gives the desired claim. 

Now suppose $u = v$, so that $p=q$ and $\ev{k}{v} = -p$. Then $\ev{k^{\pm}}{v} = -p \pm 1$ and the decoration of $v$ in $\Gamma_{+1}$ is $-p + 1$. Hence $k^- \sim 0$ by a relation of Type I. As $k^+$ is extremal, we have
\begin{equation}\label{eq:5.4}
k^+ \sim (-1)^{p-1} \otimes (k^+ - 2(v, -)_{\Gamma_{+1}}).
\end{equation}
By similar reasoning, $s^+ \sim 0$. We claim that the characteristic elements $\smash{k^+ - 2(v, -)_{\Gamma_{+1}}}$ and $s^-$ are equal. Indeed, an examination of \eqref{eq:5.3} easily shows that they agree on all vertices $w \neq v$, while we check
\[
\ev{k^+}{v} - 2(v, v)_{\Gamma_{+1}} = \ev{k}{v} + 1 - 2(-p + 1) = p - 1
\]
and
\[
\ev{s^-}{v} = \ev{s}{v} - 1 = \ev{k}{v} - 2(v,v)_{\Gamma} - 1 = p - 1.
\]
Combining this with \eqref{eq:5.4}, we thus have $k^+ \sim (-1)^{p-1} \otimes s^-$. Hence
\[
\mathbb{B}(k) = (-1/2) \otimes k^+ + (1/2) \otimes k^- \sim (-1/2)(-1)^{p-1} \otimes s^-.
\]
Since
\[
\mathbb{B}(s) = (-1/2) \otimes s^+ + (1/2) \otimes s^- \sim (1/2) \otimes s^-
\]
this establishes the desired relation between $\mathbb{B}(k)$ and $\mathbb{B}(s)$. The case where $\ev{k}{u} = q$ is entirely analogous.
\end{proof}

\begin{lemma}\label{lem:5.2}
The sequence
\[
\Hl(\Gamma - v) \xrightarrow{\ \ \mathbb{A}\ \ } \Hl(\Gamma) \xrightarrow{\ \ \mathbb{B}\ \ } \Hl(\Gamma_{+1}) \longrightarrow 0
\]
is exact.
\end{lemma}
\begin{proof}
We begin by showing that $\mathbb{B}$ is surjective. Let $k$ be an element of $\Char(\Gamma_{+1})$. Let $k'_i$ be the element of $\Char(\Gamma)$ defined by
\[
\ev{k'_i}{w} = 
\begin{cases}
\ev{k}{w} & \text{for } w \neq v \\
i & \text{for } w = v
\end{cases}
\]
and consider the element of $\Hl(\Gamma)$ given by the class of
\[
2 \otimes \left (\sum^p_{\substack{i = \ev{k}{v} + 1 \\ i \equiv p \text{ mod }2}} k'_i \right)
\]
where the sum is taken over all $i$ from $\ev{k}{v} + 1$ to $p$ with $i \equiv p \text{ mod }2$. 
It is easily checked that $\mathbb{B}$ maps this element to the class of $k$, so $\mathbb{B}$ is surjective.

We now turn to exactness at $\Hl(\Gamma)$. The composition $\mathbb{B} \circ \mathbb{A}$ forms a telescoping sum which is easily seen to be zero. Hence $\im \mathbb{A} \subset \ker \mathbb{B}$, and in particular $\mathbb{B}$ descends to a map from the quotient $\Hl(\Gamma)/\im \mathbb{A}$ to $\Hl(\Gamma_{+1})$. We construct an inverse
\[
\mathbb{S} : \Hl(\Gamma_{+1}) \rightarrow \Hl(\Gamma)/\im \mathbb{A}
\] 
as follows. Let $k$ be an element of $\Char(\Gamma_{+1})$. As above, define
\begin{equation}\label{eq:5.5}
\mathbb{S}(k) = 2 \otimes \left( \sum^p_{\substack{i = \ev{k}{v} + 1 \\ i \equiv p \text{ mod }2}} k'_i \right)
\end{equation}
where the above sum is viewed as an element of $\Hl(\Gamma)/\im \mathbb{A}$. To prove $\mathbb{S}$ is a well-defined map, we must show that $\mathbb{S}$ respects relations in $\Hl(\Gamma_{+1})$. Let $u$ thus be a vertex in $\Gamma_{+1}$. Suppose that $k \sim 0$ via a relation of Type I associated to $u$. If $u \neq v$, then clearly each $k'_i \sim 0$ in $\Hl(\Gamma)$ and so $\mathbb{S}(k) = 0$. If $u = v$, however, the situation is slightly more subtle. If $\ev{k}{v} > p - 1$, then again each $k_i'$ appearing in \eqref{eq:5.5} is zero in $\Hl(\Gamma)$. If $\ev{k}{v} < - (p - 1)$, however, then \eqref{eq:5.5} is \textit{not} zero in $\Hl(\Gamma)$, and is instead given by the class of
\begin{equation}\label{eq:5.6}
2 \otimes \left( \sum^p_{\substack{i = -p \\ i \equiv p \text{ mod }2}} k'_i \right).
\end{equation}
Thus, $\mathbb{S}$ does \textit{not} give a well-defined map from $\Hl(\Gamma_{+1})$ to $\Hl(\Gamma)$. Nevertheless, since \eqref{eq:5.6} is in the image of $\mathbb{A}$, we have that $\mathbb{S}(k) = 0$ in the quotient $\Hl(\Gamma) / \im \mathbb{A}$.

Similarly, it is easy to check that $\mathbb{S}$ respects relations of Type II in the case that the associated vertex $u$ is not $v$. Thus, suppose $\ev{k}{v} = -p + 1$. Write
\[
s = k - 2(v, -)_{\Gamma_{+1}},
\]
so that $k \sim (-1)^{p-1} \otimes s$. Now, using the fact that we are working in the quotient $\Hl(\Gamma)/\im \mathbb{A}$, we have
\begin{equation}\label{eq:4.7}
\mathbb{S}(k) = 2 \otimes \left(\sum^p_{\substack{i = -p + 2 \\ i \equiv p \text{ mod }2}} k'_i \right) = 2 \otimes \left(\sum^p_{\substack{i = -p + 2 \\ i \equiv p \text{ mod }2}} k'_i\right) + (-2) \otimes \left ( \sum^p_{\substack{i = -p \\ i \equiv p \text{ mod }2}} k'_i \right ) = (-2) \otimes k'_{-p}.
\end{equation}
Moreover, noting that $\ev{s}{v} = p -1$, we see that
\begin{equation}\label{eq:4.8}
\mathbb{S}(s) = 2 \otimes s_p'. 
\end{equation}
It is straightforward to verify that $k'_{-p} \sim (-1)^p \otimes s'_p$ via a relation of Type II associated to $v$. Combining this with \eqref{eq:4.7} and \eqref{eq:4.8} gives the desired relation between $\mathbb{S}(k)$ and $\mathbb{S}(s)$. The case where $\ev{k}{v} = p - 1$ is entirely analogous. 

Finally, it is straightforward to check that $\mathbb{S}$ is an inverse to $\mathbb{B}$.
This shows that $\mathbb{B}$ and $\mathbb{S}$ are isomorphisms between $\Hl(\Gamma) / \im \mathbb{A}$ and $\Hl(\Gamma_{+1})$. Hence the lattice surgery sequence is exact at $\Hl(\Gamma)$, completing the proof.
\end{proof}


\subsection{Comparison with the surgery exact sequence for $\Is$}\label{sec:5.2}
Here, we show that the maps $\Ts$ of Proposition \ref{prop:3.9} commute with the maps in the lattice and instanton surgery exact sequences, an essential ingredient in the proof of Theorem \ref{thm:1.1}. To prove this, we must take care with the multicurve decorations we are using in the latter. As in the previous subsection, let $v$ be a vertex in $\Gamma$ with decoration $-p<0$, and consider the surgery triad $Y_{\Gamma - v},Y_{\Gamma},Y_{\Gamma_{+1}}$. Let us equip $Y_{\Gamma - v}$ with the multicurve 
\[
\lambda_{\Gamma - v}' = \lambda_{\Gamma - v} \cup (p-1) \mu_v
\]
displayed on the left in Figure~\ref{fig:5.1}. This is formed by taking the standard multicurve $\lambda_{\Gamma-v}$, together with $p-1$ unknots represented by meridians of the knot $v$ in $Y_{\Gamma - v}$. (Note that in $Y_{\Gamma - v}$, we have not yet done surgery on $v$.) Following the prescription of Section~\ref{sec:2.3}, the other decorations in the surgery sequence are then easily seen to be the standard multicurve decorations $\lambda_{\Gamma}$ and $\lambda_{\Gamma_{+1}}$. As in Section~\ref{sec:2.3}, we denote the 2-handle attachment cobordisms and their surface decorations by $(W_a, \Sigma_a)$ and $(W_b, \Sigma_b)$.

\begin{figure}[t]
\center
\includegraphics[scale=0.8]{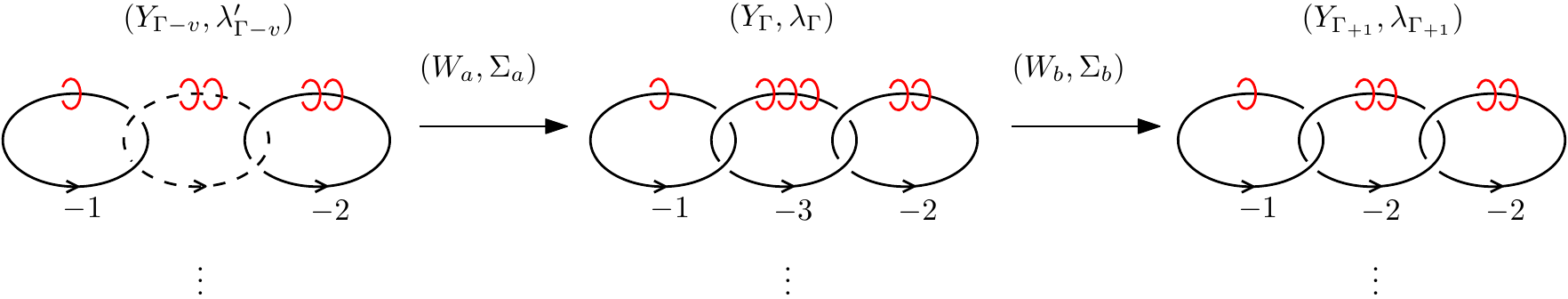}
\caption{Decorations in the surgery sequence.}\label{fig:5.1}
\end{figure}

\begin{lemma}\label{lem:5.3}
The diagram 
\[
\begin{tikzpicture}[scale=1]
\node (A) at (0,1.75) {$\Is(Y_{\Gamma-v}, \lambda_{\Gamma-v})$};
\node (B) at (1.6, 1.75) {$\cong$};
\node (C) at (3.2,1.75) {$\Is(Y_{\Gamma-v}, \lambda'_{\Gamma-v})$};
\node (D) at (7.4,1.75) {$\Is(Y_{\Gamma}, \lambda_{\Gamma})$};
\node (E) at (11.6,1.75) {$\Is(Y_{\Gamma_{+1}}, \lambda_{\Gamma_{+1}})$};
\node (F) at (0,0) {$\Hl(\Gamma-v)$};
\node (G) at (7.4,0) {$\Hl(\Gamma)$};
\node (H) at (11.6,0) {$\Hl(\Gamma_{+1})$};
\path[->,font=\scriptsize,>=angle 90]
(F) edge node[left]{$\Ts$} (A)
(G) edge node[left]{$\Ts$} (D)
(H) edge node[left]{$\Ts$} (E)
(C) edge node[above]{$\Is(W_a, \Sigma_a)$} (D)
(D) edge node[above]{$\Is(W_b, \Sigma_b)$} (E)
(F) edge node[above]{$\mathbb{A}$} (G)
(G) edge node[above]{$\mathbb{B}$} (H);
\end{tikzpicture}
\]
commutes. Here the isomorphism $\Is(Y_{\Gamma-v}, \lambda_{\Gamma - v}) \cong \Is(Y_{\Gamma-v}, \lambda_{\Gamma - v}')$ is given by the tautological cobordism map of Remark~\ref{rem:tautological}. 
\end{lemma}
\begin{proof}
We begin with the left-hand commutative square. Let $k \in \Char(\Gamma - v)$. Going up-and-then-right, this is sent to the composition of three cobordism maps applied to the generator $x_0 \in \Is(S^3, \emptyset)$ of \eqref{eq:gen}. Topologically, this composition is just
\[
W_{\Gamma - v} \cup (Y_{\Gamma - v} \times I) \cup W_a \cong W_{\Gamma}.
\]
Note that the union of the corresponding surface decorations is diffeomorphic to $D_{\Gamma}$. Hence by the composition law, going up-and-then-right in the square yields
\[
\sum_{\substack{k' \in \Char(W_{\Gamma})\\ k'|_{W_{\Gamma-v}} = k}} \Is(W_\Gamma, D_\Gamma;k')(x_0).
\]
Examining the definition of $\mathbb{A}$, this is precisely $\Ts(\mathbb{A}(k))$.

We now turn to the right-hand commutative square. Let $k \in \Char(\Gamma)$. Going up-and-then-right, this is sent to the composition of two cobordism maps applied to $x_0 \in \Is(S^3, \emptyset)$. Topologically, this composition is diffeomorphic to the blow-up
\begin{equation}\label{eq:diffeo}
W_{\Gamma} \cup W_b \cong W_{\Gamma_{+1}} \# \overline{\CPr}^2.
\end{equation}
The surface $D_\Gamma \cup \Sigma_b$ is just the standard disk system $D_{\Gamma_{+1}}$ coming from $W_{\Gamma_{+1}}$, union a single copy of the exceptional divisor $E$. Hence
\begin{align*}
\Is(W_b, \Sigma_b) \circ \Is(W_\Gamma, D_\Gamma; k) &= \sum_{\substack{k' \in \Char(W_{\Gamma} \cup W_b)\\ k'|_{W_\Gamma} = k}} \Is(W_{\Gamma_{+1}} \# \overline{\CPr}^2, D_{\Gamma_{+1}} \cup E; k') \\
&= \sum_{\substack{k' \in \Char(W_{\Gamma} \cup W_b)\\ k'|_{W_\Gamma} = k}} (-1)^{(1/2)(\ev{k'}{E} - 1)} \Is(W_{\Gamma_{+1}} \# \overline{\CPr}^2, D_{\Gamma_{+1}}; k').
\end{align*}
By the blow-up formula, there are only two terms in this sum which are non-zero, which correspond to the characteristic elements $k'$ with $\ev{k'}{E} = \pm 1$. The diffeomorphism $(\ref{eq:diffeo})$ is obtained by applying a single handleslide corresponding to the basis change $v \mapsto v - x$ (where $x = E$ is the exceptional divisor), as in Figure~\ref{fig:4.1}. The two characteristic elements $k'$ evaluate to $k(v) -1$ and $k(v) + 1$ (respectively) on this new basis element, and coincide with $k$ on all other basis elements. Hence they act as $k^{\mp}$ (respectively) on $W_{\Gamma_{+1}}$. Applying the blow-up formula, the above is thus equal to
\[
(-1/2) \cdot \Is(W_{\Gamma_{+1}}, D_{\Gamma_{+1}}; k^+) + (1/2) \cdot \Is(W_{\Gamma_{+1}}, D_{\Gamma_{+1}}; k^-).
\]
Examining the definition of $\mathbb{B}$, this is precisely $\Ts(\mathbb{B}(k))$.
\end{proof}


\section{The isomorphism theorem for plumbings with $\leq 1$ bad vertex} \label{sec:6}

Theorem~\ref{thm:1.1} asserts that the map in Proposition \ref{prop:3.9}, \[\Ts:\Hl(\Gamma)\to\Iseven(Y_\Gamma,\lambda_\Gamma)\cong \Iseven(Y_\Gamma), \] is an isomorphism for every  almost-rational plumbing. In this section, we  prove this  for negative-definite plumbings with at most one bad vertex. We will complete the proof of this isomorphism theorem---for \emph{general} almost-rational plumbings---in the next section.

Let us first consider  0-bad-vertex plumbings. Recall that these are the plumbing forests $\Gamma$ for which
\begin{equation}\label{eq:6.1}
-m(v) \geq d(v)
\end{equation}
for all vertices $v\in \Gamma$. We begin with the convenient lemma that  0-bad-vertex plumbings are \emph{automatically}  negative-semidefinite:

\begin{lemma}\label{lem:6.1}
Let $\Gamma$ be a 0-bad vertex plumbing. Then $\Gamma$ is negative-semidefinite. In fact, either:
\begin{enumerate}
\item $\Gamma$ is negative-definite; or, 
\item There is a connected component $\Gamma_0$ of $\Gamma$ such that $-m(v) = d(v)$ for all vertices $v \in \Gamma_0$. In this situation, $\Gamma_0$ necessarily represents a plumbing diagram for $S^1 \times S^2$.
\end{enumerate}
\end{lemma}
\begin{proof}
Let $x = c_1v_1 + \cdots + c_nv_n$ be an arbitrary vector in $H_2(W_\Gamma)$. Consider the simple combinatorial equality
\[
\sum_{\substack{\text{pairs }(i, j) \text{ with } i \neq j \\\text{ and } v_i \text{ and } v_j \text{ adjacent}}} (c_i + c_j)^2 = \sum_i 2c_i^2 d(v_i) + \sum_{\substack{\text{pairs }(i, j) \text{ with } i \neq j \\\text{ and } v_i \text{ and } v_j \text{ adjacent}}} 2c_ic_j.
\]
Expanding the intersection pairing and combining this with the above, we obtain
\begin{align}\label{eq:6.2}
(x, x) &= \sum_i c_i^2 m(v_i) - \sum_{\substack{\text{pairs }(i, j) \text{ with } i \neq j \\\text{ and } v_i \text{ and } v_j \text{ adjacent}}} c_ic_j \nonumber \\
&= \sum_i c_i^2 m(v_i) + \sum_i c_i^2 d(v_i) - \dfrac{1}{2} \sum_{\substack{\text{pairs }(i, j) \text{ with } i \neq j \\\text{ and } v_i \text{ and } v_j \text{ adjacent}}} (c_i + c_j)^2.
\end{align}
The sum of the first two terms is at most zero by \eqref{eq:6.1}, while the third term is a sum of squares. Hence,  $(x, x) \leq 0$. 

Moreover, if $(x, x) = 0$, then \eqref{eq:6.2} shows that $c_i = - c_j$ whenever $v_i$ and $v_j$ are adjacent in $\Gamma$. If $x$ is not the zero vector, this means that there is a connected component of $\Gamma$ whose vertices $v$ all have non-zero coefficients in $x$. For such vertices, \eqref{eq:6.2} furthermore shows $-m(v) = d(v)$, as desired. Any subgraph with this property is easily seen to represent $S^1 \times S^2$ by repeatedly blowing down leaves (which necessarily have framing $-1$).
\end{proof}

\begin{lemma}\label{lem:6.2}
Let $\Gamma$ be a negative-definite plumbing forest with no bad vertices. Then $Y_\Gamma$ is an instanton L-space.
\end{lemma}

\begin{proof}
We proceed by induction on the cardinality of $\Gamma$. For the graph consisting of a single vertex with decoration $-p$, the claim is evident. Thus, assume the claim has been established for all graphs with cardinality $n$, and let $\Gamma$ have cardinality $n + 1$. If any leaf in $\Gamma$ has framing $-1$, then blowing down gives a graph which still has no bad vertices and has cardinality $n$. Applying the inductive hypothesis then gives the claim. Hence, we may assume that all leaves in $\Gamma$ have framing at most $-2$.

Now, fix any leaf $v$ in $\Gamma$. We proceed by a downwards sub-induction on the framing of $v$, with the base case $m(v) = -1$ being covered by the previous paragraph. Consider the graph $\Gamma_{+1}$. We claim that this is negative-definite. Indeed, either $v$ is an isolated vertex (in which case the claim is obvious), or $v$ belongs to a connected component of $\Gamma$ with at least two leaves. In the latter case, one of these leaves still has framing at most $-2$ in $\Gamma_{+1}$. Keeping in mind that $\Gamma$ is negative-definite, Lemma~\ref{lem:6.1} then shows that $\Gamma_{+1}$ must also be negative-definite. We may thus inductively assume that $Y_{\Gamma - v}$ and $Y_{\Gamma_{+1}}$ are both instanton L-spaces. Hence \begin{equation}\label{eqn:odd}\smash{I^\#_\text{odd}(Y_{\Gamma - v}, \lambda_{\Gamma-v}) = I^\#_\text{odd}(Y_{\Gamma_{+1}}, \lambda_{\Gamma_{+1}}) = 0}.\end{equation}

We now consider  the  surgery exact triangle 
\begin{equation*}\label{eqn:triple} \dots \to \Is(Y_{\Gamma-v},\lambda_{\Gamma-v}) \xrightarrow{I^\#(W_a,\Sigma_a)} \Is(Y_\Gamma,\lambda_\Gamma) \xrightarrow{I^\#(W_b,\Sigma_b)} \Is(Y_{\Gamma_{+1}},\lambda_{\Gamma_{+1}}) \to \cdots. \end{equation*}
The maps $I^\#(W_a,\Sigma_a)$ and $I^\#(W_a,\Sigma_a)$ each have even degree. This follows from   \eqref{eqn:degree}, given that \begin{align*}
&\chi(W_a)=\chi(W_b) = 1,\\
&\sigma(W_a) = \sigma(W_b)=-1,\\
&b_1(Y_{\Gamma-v}) = b_1(Y_{\Gamma}) = b_1(Y_{\Gamma_{+1}}) = 0.\end{align*} The third claim above---that these 3-manifolds are rational homology spheres---follows from the fact that the associated plumbings are negative-definite, as discussed in Section \ref{sec:2.4}. The second claim above also follows from this fact. Indeed, the composition $W_{\Gamma-v}\cup W_a$ is diffeomorphic to $W_\Gamma$, which is  negative-definite. This implies that $W_a$ is negative-definite. Similarly,  $W_{\Gamma}\cup W_b$ is a blowup of $W_{\Gamma_{+1}}$, and is thus  negative-definite. This implies that $W_b$ is negative-definite as well.  The remaining map in the surgery exact triangle  therefore has odd degree, since the sum of the degrees of the three maps is odd \cite[Section 7.5]{Scaduto}. This final map  is then zero by \eqref{eqn:odd}, so the sequence splits and it follows immediately that $\smash{I^\#_\text{odd}(Y_{\Gamma}, \lambda_{\Gamma}) = 0}$, meaning that $Y_\Gamma$ is an instanton L-space, as desired.\end{proof}

\begin{remark}
One can show that if $\Gamma$ is a negative-definite plumbing forest with no bad vertices, then $Y_\Gamma$ is the branched double cover of $S^3$ along a quasi-alternating link. 
\end{remark}

We  may now  prove Theorem \ref{thm:1.1} when $\Gamma$ has no bad vertices. 


\begin{lemma}\label{lem:6.3}
Let $\Gamma$ be a negative-definite plumbing forest with no bad vertices. Then the map \[\Ts: \Hl(\Gamma)\to I^\#_\text{even}(Y_\Gamma) = \Is(Y_\Gamma)\] is an isomorphism. 
\end{lemma}

\begin{proof}
We proceed by the same induction as in Lemma~\ref{lem:6.2}. For the graph consisting of a single vertex with decoration $-p$, the claim is  Proposition~\ref{prop:4.4}. If any leaf in $\Gamma$ has framing $-1$, then blowing down along this leaf and using the naturality of lattice homology and the map $\Ts$ under blow-downs (see Appendix~\ref{sec:appendix}) proves the claim. Indeed, if $\tilde\Gamma$ is the result of blowing down $\Gamma$ along this leaf (deleting the leaf from $\Gamma$ and increasing the framing of its neighbor by one), then it follows from Lemma \ref{lem:A.2} that there is a commutative square 
\[
\xymatrix@C=3.5em{
& \Is(Y_{\Gamma}, \lambda_{\Gamma}) \ar[r]^{\cong} & \Is(Y_{\tilde\Gamma}, \lambda_{\tilde\Gamma}) \\
& \Hl(\Gamma) \ar[r]^{\cong} \ar[u]^{\Ts} & \Hl(\tilde\Gamma) \ar[u]^-{\Ts}
}
\]
By induction, the rightmost vertical map is an isomorphism; therefore, so is the leftmost.

Now fix any leaf $v$. As in Lemma \ref{lem:6.2}, we proceed by downward sub-induction on the framing of $v$. The base case $m(v)=-1$ is covered above. We may  therefore assume, as in the proof of Lemma \ref{lem:6.2}, that $\Gamma-v$ and $\Gamma_{+1}$ are both negative-definite with no bad vertices, which implies that $Y_{\Gamma-v}$ and $Y_{\Gamma_{+1}}$ are both instanton L-spaces, by Lemma \ref{lem:6.2}. 
We thus have the commutative diagram below, by Lemma~\ref{lem:5.3}:
\[
\xymatrix{
0 \ar[r] & I^\#_\text{even}(Y_{\Gamma-v}, \lambda_{\Gamma-v}) \ar[r] & I^\#_\text{even}(Y_{\Gamma}, \lambda_{\Gamma})\ar[r] & I^\#_\text{even}(Y_{\Gamma_{+1}}, \lambda_{\Gamma_{+1}}) \ar[r] &0 \\ 
& \Hl(\Gamma-v) \ar[r]^{\mathbb{A}} \ar[u]^-{\Ts}&\Hl(\Gamma) \ar[r]^{\mathbb{B}} \ar[u]^{\Ts} & \Hl(\Gamma_{+1}) \ar[r] \ar[u]^-{\Ts} & 0 \\}
\]
The top row is  the surgery exact triangle while the bottom row is exact by Lemma~\ref{lem:5.2}. That the middle two   maps in the top row have even degree while the rightmost map has odd degree follows from the fact that  $\Gamma-v$, $\Gamma$, and $\Gamma_{+1}$ are  negative-definite, just as in the proof of Lemma \ref{lem:6.2}. We may inductively assume that   the leftmost and rightmost vertical maps are isomorphisms; hence, the middle vertical map is also an isomorphism, by the five-lemma.
\end{proof}

A slight extension of this argument proves Theorem~\ref{thm:1.1} when $\Gamma$ has at most one bad vertex. Recall that this means there is at most one vertex $v$ in $\Gamma$ which violates the inequality \eqref{eq:6.1}.

\begin{lemma}\label{lem:1bad}
Let $\Gamma$ be a negative-definite plumbing forest with one bad vertex. Then the map \[\Ts: \Hl(\Gamma)\to I^\#_\text{even}(Y_\Gamma)\] is an isomorphism. 
\end{lemma}

\begin{proof}
Let $\Gamma$ be a negative-definite plumbing forest with a single bad vertex $v$. We proceed by upwards induction on the framing of $v$. Note that decreasing the framing of $v$ results in a graph that is still negative-definite, and decreasing the framing sufficiently results in a plumbing with no bad vertices. The map $\Ts$ for this  0-bad-vertex plumbing is   an isomorphism by Lemma~\ref{lem:6.3}. By induction, it is therefore enough to assume that $\Ts$ is an isomorphism for  the plumbing $\Gamma_{-1}$ obtained by decreasing the framing of $v$ by one, and prove that $\Ts$ is an isomorphism for $\Gamma$.  By Lemma~\ref{lem:5.3}, we have    the commutative diagram 
\[
\xymatrix{
& I^\#_\text{even}(Y_{\Gamma-v}, \lambda_{\Gamma-v}) \ar[r] & I^\#_\text{even}(Y_{\Gamma_{-1}}, \lambda_{\Gamma_{-1}})\ar[r] & I^\#_\text{even}(Y_{\Gamma}, \lambda_{\Gamma}) \ar[r] &0 \\ 
& \Hl(\Gamma-v) \ar[r]^{\mathbb{A}} \ar[u]^-{\Ts}&\Hl(\Gamma_{-1}) \ar[r]^{\mathbb{B}} \ar[u]^-{\Ts} & \Hl(\Gamma) \ar[r] \ar[u]^-{\Ts} & 0. \\}
\]
That the two leftmost  maps in the top row have even degree while the rightmost map has odd degree follows from the fact that  $\Gamma-v$, $\Gamma_{-1}$, and $\Gamma$ are all negative-definite, just as in the proofs of Lemmas \ref{lem:6.2} and \ref{lem:6.3}. The map in the top right is zero since $\Gamma - v$ has no bad vertices (and is negative-definite), which implies that $Y_{\Gamma - v}$ is  an instanton L-space, by Lemma \ref{lem:6.2}.  The  two leftmost vertical maps are isomorphisms, by our inductive assumption and Lemma~\ref{lem:6.3}. The rightmost vertical map is then also an isomorphism, by the five-lemma. \end{proof}

\begin{remark}The above proof cannot be extended to plumbings with two bad vertices. Indeed, if $v$ is a bad vertex in such a graph, then $\Gamma - v$ has one bad vertex. In this case, $Y_{\Gamma - v}$ may \textit{not} be an instanton L-space, and so the top right map in the above commutative diagram need not be zero.
\end{remark}


\section{The isomorphism theorem for almost-rational plumbings}\label{sec:7} In this section, we first explain how the construction of lattice homology given in Definition \ref{def:3.1} is related to the version of lattice cohomology defined in \cite{Nemethi}, and then use this relationship together with work of N{\'e}methi  to finish the proof of Theorem \ref{thm:1.1}. We then  prove Corollaries~\ref{cor:1.2}, \ref{cor:1.2a}, and \ref{cor:1.3} using the relationship between lattice cohomology and Heegaard Floer homology.

\subsection{Sign conventions in lattice homology}\label{sec:7.1}
The first step in relating $\Hl(\Gamma)$ with the lattice cohomology theory in \cite{Nemethi} is to reconcile different sign conventions in the pairing associated to the plumbing forest $\Gamma$ and in the Type II relation of Definition \ref{def:3.1}. 

In our convention, adjacent vertices of $\Gamma$ have intersection pairing $-1$; as mentioned in Section \ref{sec:2.4}, this corresponds to labelling the edges of $\Gamma$ by $-1$. Let $\smash{\overline{\Gamma}}$ be the graph obtained from $\Gamma$ by changing the edge labellings to $+1$, so that adjacent vertices of $\smash{\overline{\Gamma}}$ have pairing $+1$. We can then define the lattice homology $\Hl(\smash{\overline{\Gamma}})$ as the quotient of \[V_{\smash{\overline{\Gamma}}} = \mathbb{C}\otimes\Char(\smash{\overline{\Gamma}})\]  by the  relations in Definition \ref{def:3.1}. We claim that that $\Hl(\Gamma)$ and $\Hl(\smash{\overline{\Gamma}})$ are isomorphic. Indeed, forests are bipartite, so we can partition the vertices of $\Gamma$ into two sets $V$ and $V'$ so that no two vertices of $V$ or of $V'$ are adjacent. The change in pairing convention from $\Gamma$ to $\smash{\overline{\Gamma}}$ corresponds to reversing the orientation of each 2-sphere corresponding to a vertex of $V$. For each $k\in \Char(\Gamma),$ let $\overline{k}\in\Char(\smash{\overline{\Gamma}})$ be the characteristic vector obtained from $k$ by negating its evaluation on each element of $V$. The linear isomorphism from $V_\Gamma$ to $V_{\smash{\overline{\Gamma}}}$ defined by \[1\otimes k \mapsto 1\otimes \overline{k}\] respects the relations of Definition \ref{def:3.1}, and  therefore induces an isomorphism \[\Hl(\Gamma)\to\Hl(\smash{\overline{\Gamma}}),\] as claimed. 

We next show that omitting the factor of $(-1)^{v^2}$ appearing in the  Type II relation of Definition \ref{def:3.1} yields an isomorphic version of lattice homology, which we denote by $\Hlb(\Gamma)$.

Recall from Section \ref{sec:2.4} that for  $k\in\Char(\Gamma)$ we denote by $[k]$ the orbit of $k$ under the action of $H_2(W_\Gamma)$ on $\Char(\Gamma)$, where $x\in H_2(W_\Gamma)$ acts by \[k\mapsto k+2x^*.\] These orbits partition $\Char(\Gamma)$. Given an orbit $[k_0]$ and some $k \in [k_0]$,  write \[k = k_0 + \sum_i \left( c_i \cdot 2v_i^* \right),\] and define
\[
\sigma(k) = \sum_{i \text{ with } v_i^2 \text{ odd}} c_i.
\]
Then the linear automorphism of $\C \otimes \Char(\Gamma)$ defined by sending \[1 \otimes k\mapsto (-1)^{\sigma(k)} \otimes k\] sends the Type II relations of Definition \ref{def:3.1} to relations which are the same except without the factor of  $\smash{(-1)^{v^2}}$. The automorphism above then induces an isomorphism \[\Hl(\Gamma)\to\Hlb(\Gamma).\] Combining this with the isomorphism $\Hl(\Gamma) \cong \Hl(\smash{\overline{\Gamma}})$, we have that \begin{equation}\label{eqn:Hlb}\Hl(\Gamma) \cong \Hlb(\smash{\overline{\Gamma}}).\end{equation}
In summary, we are free to (a) use the  standard $+1$ sign convention for the pairing of adjacent vertices of $\Gamma$, and (b) omit the factor of $\smash{(-1)^{v^2}}$ appearing in the Type II relation of  Definition~\ref{def:3.1}. The result  is a $\C$-vector space which is isomorphic to  $\Hl(\Gamma)$ as originally defined. It is this  slightly altered version that is most closely related to the lattice cohomology described next.

\subsection{Lattice cohomology and Heegaard Floer homology}\label{sec:7.2}
We now provide a brief review of the construction of lattice cohomology due by N\'emethi in \cite{Nemethi}, which is itself a reformulation of  Ozsv\'ath and Szab\'o's original construction in  \cite{OSplumbed}. 

Let $\Gamma$ be a negative-definite plumbing forest. Fix an element $k_0\in \Char(\Gamma)$. 
View $\smash{H_2(W) \cong \Z^{|\Gamma|}}$ as the lattice generated by  the vertices of $\Gamma$. Define a \textit{weight function} on this lattice by setting
\[
w(x) = - ((x, x) + \ev{k_0}{x})/2,
\] where the convention in the pairing $(-,-)$ above is that adjacent vertices in $\Gamma$ pair to $+1$ (in contrast with our convention; see Section \ref{sec:7.1}).
For $0\leq q \leq|\Gamma|$ consider the set of $q$-dimensional unit cubes in $\mathbb{R}^{|\Gamma|}$ which have vertices lying in $\mathbb{Z}^{|\Gamma|} \subset \mathbb{R}^{|\Gamma|}$. For any such cube $\Box_q$, set
\[
w(\Box_q) = \max_{x \text{ a vertex of } \Box_q} \{w(x)\}.
\] 
For any integer $n$, define the \textit{sublevel set} $S_n$ of weight $n$ to be the cubical subcomplex of $\mathbb{R}^{|\Gamma|}$ formed by taking the union over all cells (of all dimensions) with weight  less than or equal to $n$:
\[
S_n = \bigcup_{w(\Box_q) \leq n} \Box_q.
\]
In \cite[Definition 3.1.11]{Nemethi}, N\'emethi defines the lattice cohomology by setting
\[
\mathbb{H}^+(\Gamma, [k_0]) = \bigoplus_{n} H^0(S_n; \Z)
\]
where $H^0(S_n; \Z)$ is the  zeroth singular cohomology of $S_n$ with $\Z$-coefficients.\footnote{We can  extend the definition of lattice cohomology to include the higher cohomology groups of $S_n$; see \cite[Definition 3.1.11]{Nemethi}. However, we will not do this here.} Note that this is just the $\Z$-span of the connected components of $S_n$. 
The inclusion $i:S_{n-1} \to S_n$ induces a map
\[
i^*: H^0(S_n; \Z) \rightarrow H^0(S_{n-1}; \Z),
\]
and we define an action of $\Z[U]$ on $\mathbb{H}^+(\Gamma, [k_0])$ where $U$ acts by $i^*$. As suggested by the notation, N{\'e}methi shows in \cite{Nemethi} that $\mathbb{H}^+(\Gamma, [k_0])$ does not depend on the choice of representative in $[k_0]$. Moreover, he proves the following:

\begin{theorem}\cite[Theorem 5.2.2]{Nemethi}\label{thm:7.1}
 If $\Gamma$ is almost-rational, then \[\HFpo(-Y_\Gamma)=0,\] and there is an isomorphism of $\Z[U]$-modules,
\[
 \HFpe(-Y_\Gamma, [k_0]) \cong \mathbb{H}^+(\Gamma, [k_0]).
\] 
\end{theorem}

In the theorem above, $\HFp(-Y_\Gamma, [k_0])$ is the summand of $\HFp(-Y_\Gamma)$ in the $\textrm{spin}^c$ structure  corresponding to the orbit $[k_0]$, per Remark \ref{rmk:spinc}. One forms the \emph{total} lattice cohomology $\mathbb{H}^+(\Gamma)$ by summing the groups $\mathbb{H}^+(\Gamma, [k_0])$ over all distinct orbits $[k_0]$. We then have that 
\[\HFpe(-Y_\Gamma)\cong \mathbb{H}^+(\Gamma).\]
\noindent
In fact, the total lattice homology $\HFp(-Y_\Gamma)$ comes with a natural $\Z$-grading, which turns Theorem~\ref{thm:7.1} into an isomorphism of absolutely $\Z$-graded $\Z[U]$-modules (up to an overall grading shift). For further discussion, see \cite{Nemethi}.

\begin{remark} \label{torsion}
Note that  $\mathbb{H}^+(\Gamma)$ has no $\Z$-torsion since it is just a direct sum of zeroth cohomology groups. Therefore, $\HFp(-Y_\Gamma)$ has no $\Z$-torsion for almost-rational plumbings $\Gamma$, by Theorem \ref{thm:7.1}.
\end{remark}

We next address the relationship between our version of lattice homology and the version of  lattice cohomology defined by N{\'e}methi as above. First, note that while we defined $\Hl(\Gamma)$ over $\C$ in order to relate it to framed instanton  homology, we could just as easily have defined it over $\Z$. Let $\Hl(\Gamma;\Z)$ denote the combinatorial theory defined as in Definition \ref{def:3.1} but over the integers.

\begin{lemma}
\label{lem:isom}  There is an isomorphism of abelian groups, \[\Hl(\Gamma;\Z)\cong \ker U \subset \mathbb{H}^+(\Gamma).\]
\end{lemma} 

\begin{proof}
First note that  we can write \[\Hl(\Gamma;\Z)=\bigoplus_{\textrm{distinct orbits } [k_0]} \Hl(\Gamma,[k_0];\Z),\] where  $\Hl(\Gamma,[k_0];\Z)$ is the quotient of the free abelian group generated by  $k\in[k_0]$ by the relations in Definition \ref{def:3.1}. Now fix an arbitrary $k_0\in\Char(\Gamma)$. It suffices to show that \[\Hl(\Gamma,[k_0];\Z)\cong \ker U \subset \mathbb{H}^+(\Gamma,[k_0]).\] Let us first understand the right-hand side.

We say that a lattice point $x\in H_2(W_\Gamma)$ is a \textit{local minimum of $w$} if $w(x) \leq w(x')$ for all lattice points $x'$ adjacent to $x$. Then it is not hard to see that $\ker U\subset \mathbb{H}^+(\Gamma,[k_0])$ is generated by local minima of $w$ (really, by elements of the form $1\otimes x$ where $x$ is a local minimum, but we will not worry about this distinction), with the one relation that if $x$ and $x'$ are two adjacent lattice points with the same weight, then $x \sim x'$. 

We will show that the identification of $H_2(W_\Gamma)$ with $[k_0] \subset \Char(\Gamma)$ which sends $x$ to $k_0+2x^*$ induces an isomorphism \begin{equation}\label{keriso}\ker U\subset \mathbb{H}^+(\Gamma,[k_0]) \to \Hl(\Gamma,[k_0];\Z).\end{equation} Let $x\in H_2(W_\Gamma)$ and suppose $x'=x+v$ is an adjacent element in the lattice. Then $x$  and $x'$ are identified with the elements $k=k_0 + 2x^*$ and $k+2v^*$ in the orbit $[k_0]$. Moreover, we have 
\begin{align*}
w(x + v) - w(x) &= -((x + v, x + v) + \ev{k_0}{x + v})/2 + ((x, x) + \ev{k_0}{x})/2 \\
&= -(2(x, v) + (v, v) + \ev{k_0}{v})/2 \\
&= -(\ev{k}{v} + v^2)/2.
\end{align*}In particular,  $\ev{k}{v} > - v^2$ iff $w(x + v) < w(x)$, which implies that $x$ is not a local minimum.  A similar computation shows that $\ev{k}{v} <  v^2$ iff $w(x - v) < w(x)$. This shows, by the Type I relation in Definition \ref{def:3.1}, that the local minima of $w$ generate $\Hl(\Gamma,[k_0];\Z)$ under the identification above. To prove that \eqref{keriso} is an isomorphism, it  only remains to show that the relation $x\sim x'$ is identified with the Type II relation of Definition \ref{def:3.1}. But this follows from the computation above, which shows that $w(x+v)=w(x)$ if and only if $\ev{k}{v} = - v^2$ and $w(x-v)=w(x)$ if and only if $\ev{k}{v} =  v^2$.

Note that we are using  the alternative conventions for $\Hl(\Gamma)$ discussed in Section \ref{sec:7.1}, in which adjacent vertices of $\Gamma$ have pairing $+1$, and there is no factor of $\smash{(-1)^{v^2}}$ in the Type II relation.
\end{proof}

\begin{remark}
As observed in the proof of Lemma \ref{lem:isom}, there is a natural decomposition of $\Hl(\Gamma)$ along $\textrm{spin}^c$ structures on $Y_\Gamma$. This gives rise to a $\textrm{spin}^c$ decomposition of $\Iseven(Y_\Gamma)$ via the isomorphism in Theorem \ref{thm:1.1}. We do not prove here that this decomposition is well-defined (i.e. independent of $\Gamma$), though we expect that it is. If so, then it is notable for the fact that there is no such $\textrm{spin}^c$ decomposition of framed instanton homology groups in general.
\end{remark}


\subsection{The instanton isomorphism for almost-rational plumbings}

We now complete the proof of Theorem~\ref{thm:1.1}, for all almost-rational plumbings. The first step is the following analogue of Lemmas~\ref{lem:6.2} and \ref{lem:6.3} for rational plumbings.

\begin{lemma}\label{lem:7.2}
Let $\Gamma$ be a rational plumbing forest. Then $Y_{\Gamma}$ is an instanton L-space, and the map \[\Ts: \Hl(\Gamma)\to I^\#_\text{even}(Y_\Gamma)=\Is(Y_\Gamma)\] is an isomorphism. 
\end{lemma}
\begin{proof}
We follow the proof  in \cite[Theorem 8.3]{NemethiOS}. If $\Gamma$ is a rational plumbing, then deleting a vertex from $\Gamma$ or decreasing the framing of any vertex also yields a rational graph. Moreover, every negative-definite 0-bad-vertex plumbing is rational (see \cite[Section 6.2]{NemethiOS}). Hence, we proceed by upwards induction on the framings of the vertices, with the class of negative-definite 0-bad-vertex plumbings as the base case. In contrast with the proof of Lemma~\ref{lem:6.3}, however, we will need the following   computation due to N\'emethi: by \cite[Theorem 6.3]{NemethiOS}, if $\Gamma$ is a rational graph, then \[\mathbb{H}^+(\Gamma,[k_0]) \cong \Z[U,U^{-1}]/\Z[U]\] for each orbit $[k_0]$. It follows that \[\Z\cong \ker U \subset \mathbb{H}^+(\Gamma, [k_0]) \] for each of the $|H_1(Y_\Gamma)|$ orbits. By Lemma \ref{lem:isom}, this implies that  \[\dim\Hl(\Gamma)=|H_1(Y_\Gamma)|.\] Hence, if $\Gamma$ is a rational graph, then showing that $\smash{\Ts: \Hl(\Gamma)\to I^\#_\text{even}(Y_\Gamma)}$ is an isomorphism in fact proves that $Y_\Gamma$ is an instanton L-space.

Now, let $\Gamma$ be a rational graph, and assume  that the lemma has been established for $\Gamma - v$ and $\smash{\Gamma_{-1}}$ (as defined in the proof of Lemma~\ref{lem:1bad}). Consider again the commutative diagram of Lemma \ref{lem:5.3}:
\[
\xymatrix{
& I^\#_\text{even}(Y_{\Gamma-v}, \lambda_{\Gamma-v}) \ar[r] & I^\#_\text{even}(Y_{\Gamma_{-1}}, \lambda_{\Gamma_{-1}})\ar[r] & I^\#_\text{even}(Y_{\Gamma}, \lambda_{\Gamma}) \ar[r] &0 \\ 
& \Hl(\Gamma-v) \ar[r]^{\mathbb{A}} \ar[u]^{\Ts}&\Hl(\Gamma_{-1}) \ar[r]^{\mathbb{B}} \ar[u]^{\Ts} & \Hl(\Gamma) \ar[r] \ar[u]^{\Ts} & 0 \\}
\]
The rightmost map in the top row is zero because  $Y_{\Gamma - v}$ is an instanton L-space (as explained above) and because the map has odd degree (because all three plumbings are negative-definite, as in the proof of Lemma \ref{lem:6.2}). The two leftmost vertical maps are isomorphisms, by assumption, so the rightmost vertical map is also an isomorphism, by the five-lemma, completing the proof by induction.
\end{proof}

We may now complete the proof of Theorem \ref{thm:1.1}.

\begin{proof}[Proof of Theorem~\ref{thm:1.1}] Recall that a plumbing $\Gamma$ is almost-rational if there exists a single vertex of $\Gamma$ such that decreasing the framing on this vertex yields a rational plumbing. The proof that \[\Ts: \Hl(\Gamma)\to I^\#_\text{even}(Y_\Gamma)\] is an isomorphism for  almost-rational plumbings thus follows from Lemma~\ref{lem:7.2} in exactly the same way that the proof for negative-definite 1-bad-vertex plumbings (Lemma~\ref{lem:1bad}) followed from the result for negative-definite 0-bad-vertex plumbings (Lemma~\ref{lem:6.3}).
\end{proof}

We end this section by proving Corollaries~\ref{cor:1.2}, \ref{cor:1.2a}, and \ref{cor:1.3}.

\begin{proof}[Proof of Corollary~\ref{cor:1.2}]
Let $\Gamma$ be an almost-rational plumbing. From Lemma \ref{lem:isom}, we have that  \[\Hl(\Gamma)\cong \ker U \subset \C\otimes \mathbb{H}^+(\Gamma).\] Moreover, Theorem \ref{thm:7.1} implies that \[\C\otimes \mathbb{H}^+(\Gamma)\cong \HFpe(-Y_\Gamma;\C) = \HFp(-Y_\Gamma;\C).\] In the long exact sequence \[\dots \to\HFhat(-Y_\Gamma;\C)\to \HFp(-Y_\Gamma;\C) \xrightarrow{\cdot U} \HFp(-Y_\Gamma;\C) \to   \dots,\] the  middle two maps have even  degree while the rest have odd degree. Since $\HFp(-Y_\Gamma;\C)$ is supported in even grading, it follows that \[\Hl(\Gamma)\cong \ker U \cong \HFhat_{\textrm{even}}(-Y_\Gamma;\C)\cong \HFhat_{\textrm{even}}(Y_\Gamma;\C).\] Therefore, \[\Iseven(Y_\Gamma)\cong\HFhat_{\textrm{even}}(Y_\Gamma;\C)\] by Theorem \ref{thm:1.1}. Finally, since $\Is(Y_\Gamma)$ and $\HFhat(Y_\Gamma)$ both have Euler characteristic $|H_1(Y_\Gamma)|$, it follows that  \[\Is(Y_\Gamma)\cong\HFhat(Y_\Gamma;\C)\] as $\Z/2\Z$-graded complex vector spaces, proving the corollary.
\end{proof}

\begin{proof}[Proof of Corollary~\ref{cor:1.2a}] Suppose  that $Y$ is a Seifert fibered rational homology sphere. Then $Y$ has base either $S^2$ or $\mathbb{RP}^2$. If the base is $S^2$, then (after possibly reversing  the orientation of $Y$)  $Y\cong Y_{\Gamma}$ where $\Gamma$ is some almost-rational (in fact, negative-definite 1-bad-vertex) plumbing. In this case, the $\Z/2\Z$-graded isomorphism \begin{equation}\label{eqn:isosf}\Is(Y)\cong\HFhat(Y;\C)\end{equation} follows from Theorem \ref{thm:1.1}. 

If instead the base orbifold is $\mathbb{RP}^2(\alpha_1,\dots,\alpha_k)$ for some $k\geq 1$,  one can show that $Y$ is an instanton L-space by mimicking Boyer--Gordon--Watson's inductive proof in \cite[Proposition~18]{BGW} that $Y$ is a Heegaard Floer L-space over $\Z/2\Z$ (and therefore over $\Z$ and $\C$). Indeed, in the base case $k=1$, they observe that  $Y$ is either $\mathbb{RP}^3 \# \mathbb{RP}^3$, which is both an instanton and Heegaard Floer L-space, or that $Y$ is an elliptic manifold, having a Seifert fibration with base $S^2(2,2,n)$ for some $n$. In the latter case, $Y$ is a Heegaard Floer L-space and, by the  isomorphism \eqref{eqn:isosf} when the base is $S^2$, we conclude that $Y$ is also an instanton L-space. Boyer--Gordon--Watson then induct on $k$, using nothing from Heegaard Floer homology beyond the surgery exact triangle and the Euler characteristic formula for $\HFhat$ of rational homology spheres. Their inductive argument thus translates verbatim to the framed instanton setting.
\end{proof}

\begin{proof}[Proof of Corollary \ref{cor:1.3}] Let $Y$ be a Seifert fibered rational homology sphere. If the base is $\mathbb{RP}^2$ then $Y$ is an instanton L-space, as in the proof of Corollary \ref{cor:1.2a}, and $\pi_1(Y)$ is not left-orderable, by work of Boyer--Rolfsen--Wiest \cite[Theorem~1.3]{BRW}.
If the base is $S^2$ instead, then $Y\cong Y_\Gamma$ for some almost-rational plumbing $\Gamma$,  and so we have a chain of equivalences:
\begin{align*}
Y\text{ is an instanton L-space} &\Longleftrightarrow Y\text{ is a Heegaard Floer L-space over }\C \\
&\Longleftrightarrow Y\text{ is a Heegaard Floer L-space over }\Z/2\Z \\
&\Longleftrightarrow Y\text{ has no horizontal foliations} \\
&\Longleftrightarrow \pi_1(Y)\text{ is not left-orderable}.
\end{align*}
The first equivalence is Theorem \ref{thm:1.1}. The second equivalence comes from the fact that $\HFp(-Y_\Gamma)$ has no 2-torsion, per Remark \ref{torsion}. The third and fourth equivalences are due to Lisca--Stipsicz \cite[Theorem~1.1]{LiscaStipsicz3} and Boyer--Rolfsen--Wiest \cite[Theorem~1.3]{BRW}, respectively.
\end{proof}

\section{Examples and applications}\label{sec:8}
Let $Y$ be a rational homology sphere which is not an instanton L-space. One might then hope that there exists an irreducible homomorphism $\pi_1(Y)\to SU(2)$. 
 Unfortunately, as discussed in Section~\ref{sec:1.1}, this is not known to  hold in general.  Instead, we have the following result:
\begin{theorem}\cite[Theorem 4.6]{BSstein}
\label{thm:irrep}
Let $Y$ be a rational homology sphere which is not an instanton L-space. If $\pi_1(Y)$ is cyclically finite, then there exists an irreducible homomorphism $\pi_1(Y) \rightarrow SU(2)$.
\end{theorem}

We will not define  cyclical finiteness here; see the discussion surrounding \cite[Theorem 4.6]{BSstein} for details. However, we do note that if the universal abelian cover of $Y$ is a rational homology sphere, then $\pi_1(Y)$ is cyclically finite. Hence, we immediately obtain a proof of Corollary~\ref{cor:1.4}.
\begin{proof}[Proof of Corollary~\ref{cor:1.4}]
We are assuming that $\Gamma$ is almost-rational. Therefore, since $Y_\Gamma$ is not a Heegaard Floer L-space,  it is also not an instanton L-space, by Corollary~\ref{cor:1.2}. Since the universal abelian cover of $Y_\Gamma$ is  a rational homology sphere, $\pi_1(Y_\Gamma)$ is cyclically finite per the discussion above, and so there exists an irreducible homomorphism $\pi_1(Y_\Gamma)\to SU(2)$, by Theorem \ref{thm:irrep}.
\end{proof}

As mentioned in Section~\ref{sec:1.1}, Pedersen gives a complete combinatorial characterization of which graph manifolds have rational homology sphere universal abelian covers \cite[Theorem 1.3]{Pedersen}. This proceeds by first extracting an auxiliary graph (called the \textit{splice diagram}) from the plumbing diagram; see the discussion in \cite[Section 2.2]{Pedersensplice}. The desired characterization is then expressed in terms of various decorations on the splice diagram; see \cite[Definition 4.1]{Pedersen}. 

Below, we use Pedersen's characterization to give an example of a rational homology sphere satisfying the hypotheses of Corollary~\ref{cor:1.4}.

\begin{example}
Let $Y$ be the boundary of the plumbing given in Figure~\ref{fig:8.1}. An examination of the resulting splice diagram and \cite[Definition 4.1]{Pedersen} shows that the universal abelian cover of $Y$ is a rational homology sphere. (There are certainly other ways one can prove this fact, but we give this example here to demonstrate the method.)
We claim that $Y$ is the boundary of an almost-rational plumbing (this is not immediate from the plumbing diagram here) and is not an instanton L-space.  For this, we quote a  result of N\'emethi. According to \cite[Table 1]{Laufer} and the discussion preceding it, $Y$ is the link of an elliptic singularity.\footnote{The singularity  is given by  $z^2 = (y + x^3)(y^2 + x^7)$. See the entry for $A_{n, ****}$ in \cite[Table 1]{Laufer}.} By \cite[Proposition~4.2.2]{Nemethi}, $Y$ is then the boundary of an almost-rational plumbing for which \[ \Hl(\Gamma;\Z)\cong\ker U\subset \mathbb{H}^+(\Gamma)\] has rank  greater than $|H_1(Y)|$. Therefore, $Y$ is not an instanton L-space, by Theorem \ref{thm:1.1}, and thus it admits an irreducible homomorphism $\pi_1(Y)\to SU(2)$, by Theorem \ref{thm:irrep} and the discussion above.
\end{example}

\begin{figure}[t]
\center
\includegraphics[scale=0.9]{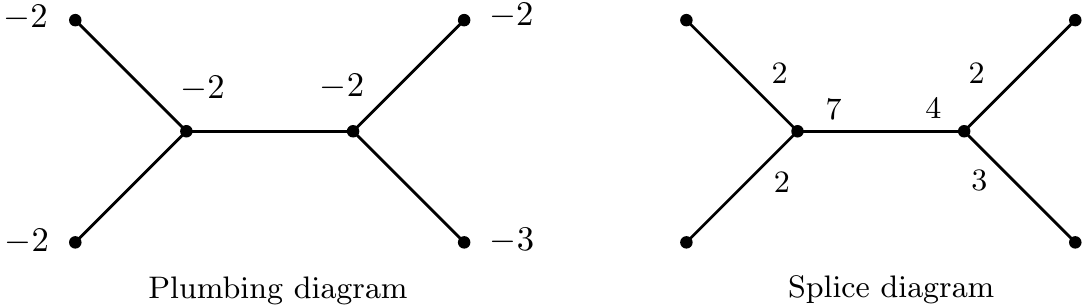}
\caption{Left: plumbing diagram for $Y$. Right: the resulting splice diagram, for use in applying \cite[Theorem 1.3]{Pedersen}.}\label{fig:8.1}
\end{figure}

Note that the manifold $Y$ described by Figure~\ref{fig:8.1} is not Seifert fibered, according to results of Neumann \cite{Neumann}.  To see this, we apply the algorithm of \cite[\S4]{Neumann} to the plumbing diagram of Figure~\ref{fig:8.1} to put it in ``normal form'':
\[ \begin{tikzpicture}[every path/.style={thick},endpoint/.style={circle,fill=black,minimum width=4pt,inner sep=0,outer sep=0}]
\begin{scope}
\node[endpoint,label={180:\scriptsize$-2$}] (a) at (0,0) {};
\node[endpoint,label={180:\scriptsize$-2$}] (c) at (135:1) {};
\node[endpoint,label={180:\scriptsize$-2$}] (d) at (225:1) {};
\node[endpoint,label={0:\scriptsize$-2$}] (b) at (1.5,0) {};
\node[endpoint,label={0:\scriptsize$-2$}] (e) at ($(b)+(45:1)$) {};
\node[endpoint,label={0:\scriptsize$-3$}] (f) at ($(b)+(315:1)$) {};
\draw (c) -- (a) -- (b) -- (e) (d) -- (a) (b) -- (f);
\end{scope}
\node at (3.5,0) {\Large$\leadsto$};
\begin{scope}[xshift=4.5cm]
\node[endpoint,label={90:\scriptsize$0$},label={270:\scriptsize$[-1]$}] (a) at (0,0) {};
\node[endpoint,label={90:\scriptsize$-1$}] (b) at (1,0) {};
\node[endpoint,label={0:\scriptsize$-2$}] (c) at (2,0) {};
\node[endpoint,label={0:\scriptsize$-2$}] (e) at ($(c)+(45:1)$) {};
\node[endpoint,label={0:\scriptsize$-3$}] (f) at ($(c)+(315:1)$) {};
\draw (a) -- (b) -- (c) -- (e) (c) -- (f);
\end{scope}
\node at (8.5,0) {\Large$\leadsto$};
\begin{scope}[xshift=9.5cm]
\node[endpoint,label={90:\scriptsize$1$},label={270:\scriptsize$[-1]$}] (a) at (0,0) {};
\node[endpoint,label={0:\scriptsize$-1$}] (c) at (1,0) {};
\node[endpoint,label={0:\scriptsize$-2$}] (e) at ($(c)+(45:1)$) {};
\node[endpoint,label={0:\scriptsize$-3$}] (f) at ($(c)+(315:1)$) {};
\draw (a) -- (c) -- (e) (c) -- (f);
\end{scope}
\end{tikzpicture} \]
Here we have applied the operation described in Step 3 of \cite[p.314]{Neumann}, followed by a blow-down, move R1 of \cite[p.304]{Neumann}.  (The ``$[-1]$'' indicates that the given vertex represents a disk bundle over $\RP^2$.)  Now $Y$ is not a lens space because it is not an instanton L-space, so by \cite[Corollary~5.7]{Neumann}, this is not the normal form plumbing graph of a Seifert fibered space.


\begin{example}
The one-cusped hyperbolic manifold $\texttt{m389}$ is identified by SnapPy \cite{SnapPy} as the complement of the knot $10_{139}$.  According to Dunfield's table of exceptional fillings \cite{Dunfield}, the Dehn filling
\[ \texttt{m389(0,1)} \cong \pm S^3_{13}(10_{139}) \cong \pm \Sigma_2(10_{154}) \]
is a graph manifold bounding the same plumbing tree as in Figure~\ref{fig:8.1}, except that one of the leaves on the left has weight $-3$ instead of $-2$.  This is also the link of an elliptic singularity, given in \cite[Table~2]{Laufer} by the dual graph $A_{2,****}$ with weights $A_*\cdot A_*=-2,-3,-3,-2$.\footnote{This singularity is given by the equation $z^2 = (y^2+x^3)^2+x^2y^3$.} Thus, another application of \cite[Proposition~4.2.2]{Nemethi} and Theorem \ref{thm:1.1} implies that $S^3_{13}(10_{139})$ is not an instanton L-space.  This clears up an ambiguity in \cite[Lemma~8.5]{BSconcordance}, which said that $\dim \Is(S^3_{13}(10_{139}))$ is either 13 or 15: the dimension must be 15, and $10_{139}$ is not an instanton L-space knot.
\end{example}


We end with a minor application of Corollary~\ref{cor:1.2} which may serve to illustrate how one can obtain further calculations. Let $L$ be a link in $S^3$. By work of Scaduto \cite[Theorem 1.1]{Scaduto}, there is a spectral sequence whose $E_2$ page is the reduced odd Khovanov homology of $L$ and which abuts to the framed instanton  homology of its branched double cover $\Sigma_2(L)$ (with reversed orientation). This collapses whenever the reduced odd Khovanov homology $\overline{\mathrm{Kh}}'(L)$ is \textit{thin} (for a definition and further discussion, see \cite[Section 1.2]{BSconcordance}). In \cite[Question 1.20]{BSconcordance}, Baldwin and Sivek consider the converse: is there  a link for which the reduced odd Khovanov homology is not thin, but the spectral sequence of \cite[Theorem 1.1]{Scaduto} still collapses at the $E_2$ page? In \cite[Theorem 1.21]{BSconcordance}, this was ruled out for all non-thin prime knots with at most ten crossings, with the possible exception of $10_{152}$. Here, we treat the final case $\Sigma_2(10_{152})$, which satisfies
\[ \dim \overline{\mathrm{Kh}}'(10_{152};\C) = 15, \]
and show that its spectral sequence does not collapse, as follows from:

\begin{proposition} \label{10_152}
 $\dim I^\#(\Sigma_2(10_{152})) = 13$.
\end{proposition}

\begin{proof}
The key observation is that $\Sigma_2(10_{152})$ is a Dehn filling of the one-cusped hyperbolic manifold \texttt{m038}: as verified using SnapPy and Regina \cite{Regina}, we have
\[ \Sigma_2(10_{152}) \cong \texttt{m038(-3,1)} \]
up to orientation, because both are recognized in Regina as the graph manifold
\[ \texttt{SFS [D: (2,1) (3,1)] U/m SFS [D: (2,1) (3,1)], m = [ 0,1 | 1,0 ]}. \]
Dunfield \cite{Dunfield} determined the other exceptional fillings of \texttt{m038}, which include
\begin{align*}
\texttt{m038(1,0)} &\cong L(3,1) \\
\texttt{m038(-1,1)} &\cong \texttt{SFS [S2:\ (2,1) (5,1) (5,-4)]} \\
\texttt{m038(-2,1)} &\cong \texttt{SFS [S2:\ (3,1) (4,1) (4,-3)]},
\end{align*}
and he computed that
\[
\dim_{\Z/2\Z} \HFhat(\texttt{m038(-1,1)}) = 7 \,\,\textrm{ and }\,\,
\dim_{\Z/2\Z} \HFhat(\texttt{m038(-2,1)}) = 10.
\]
As the latter two fillings are Seifert fibered rational homology spheres with base $S^2$, they bound almost-rational plumbings. Then, by Remark~\ref{torsion}, the  Heegaard Floer invariants over $\C$ have the same dimensions. By Corollary~\ref{cor:1.2}, we therefore have that   \[
\dim_{\C} \Is(\texttt{m038(-1,1)}) = 7 \,\,\textrm{ and }\,\,
\dim_{\C} \Is(\texttt{m038(-2,1)}) = 10
\] as well.
We have a surgery exact triangle of the form
\[ \dots \to I^\#(\texttt{m038(1,0)}) \xrightarrow{I^\#(W_a,\Sigma_a)} I^\#(\texttt{m038(-2,1)}) \to I^\#(\texttt{m038(-1,1)}) \to \dots, \]
with entries $\C^3$, $\C^{10}$, and $\C^7$, respectively, so that $I^\#(W_a,\Sigma_a)$ must be injective.  Likewise, we have a surgery exact triangle
\[ \dots \to I^\#(\texttt{m038(1,0)}) \to I^\#(\texttt{m038(-3,1)}) \to I^\#(\texttt{m038(-2,1)}) \xrightarrow{I^\#(W_c,\Sigma_c)} \dots, \]
and the composition
\[ I^\#(W_a,\Sigma_a) \circ I^\#(W_c,\Sigma_c): I^\#(\texttt{m038(-2,1)}) \to I^\#(\texttt{m038(-2,1)}) \]
is zero, because the cobordism $W_c \cup_{\texttt{m038(1,0)}} W_a$ contains a smoothly embedded 2-sphere of self-intersection zero which is not in the kernel of the intersection form, exactly as in \cite[Lemma~3.1]{BSconcordance}.  Since $I^\#(W_a,\Sigma_a)$ is injective, it follows that $I^\#(W_c,\Sigma_c)$ must be identically zero and so we have a splitting
\begin{align*}
I^\#(\texttt{m038(-3,1)}) &\cong I^\#(\texttt{m038(1,0)}) \oplus I^\#(\texttt{m038(-2,1)}) \\
&\cong \C^3 \oplus \C^{10} \cong \C^{13}.
\end{align*}
Thus $I^\#(\Sigma_2(10_{152})) \cong \C^{13}$, as claimed.
\end{proof}


\appendix
\section{Invariance of the group $\Hl(\Gamma)$} \label{sec:appendix}
In this section, we show that the isomorphism class of $\Hl(\Gamma)$ is invariant under blowing down a leaf with framing $-1$. More importantly, we prove that this invariance is compatible (in an appropriate sense) with the map $\Ts$. Let $\Gamma$ be a negative-definite plumbing forest, and let $v$ be a fixed vertex in $\Gamma$ with decoration $-p<0$. Consider the three related plumbings:
\begin{enumerate}
\item Define $\Gamma_{+1}$ by increasing the framing of $v$ by one, so that it has framing $-p + 1$.
\item Define $\Gamma_{+1} \cup \{x\}$ by taking the disjoint union of $\Gamma_{+1}$ with a single isolated vertex (denoted by $x$) of framing $-1$.
\item Define $\Gamma'$ by adding a new vertex (which we also denote by $x$) to $\Gamma$ with framing $-1$. This new vertex is adjacent to $v$ and no other vertices.
\end{enumerate}
See Figure~\ref{fig:A.1}. We furthermore assume that all of these new plumbings are still negative-definite. Clearly, $W_{\Gamma'}$ and $W_{\Gamma_{+1} \cup \{x\}}$ are diffeomorphic, and are both just the blow-up of $W_{\Gamma_{+1}}$. Thus, all three boundaries are the same. It is easy to see, moreover, that the multicurves $\lambda_{\Gamma'}$, $\lambda_{\Gamma_{+1} \cup \{x\}}$, and $\lambda_{\Gamma_{+1}}$ are homologous. Hence, we have canonical identifications
\[
\Is(Y_{\Gamma'}, \lambda_{\Gamma'}) \cong \Is(Y_{\Gamma_{+1} \cup \{x\}}, \lambda_{\Gamma_{+1} \cup \{x\}}) \cong \Is(Y_{\Gamma_{+1}}, \lambda_{\Gamma_{+1}}),
\]
as described in Remark~\ref{rem:tautological}.

\begin{figure}[t]
\center
\includegraphics[scale=0.8]{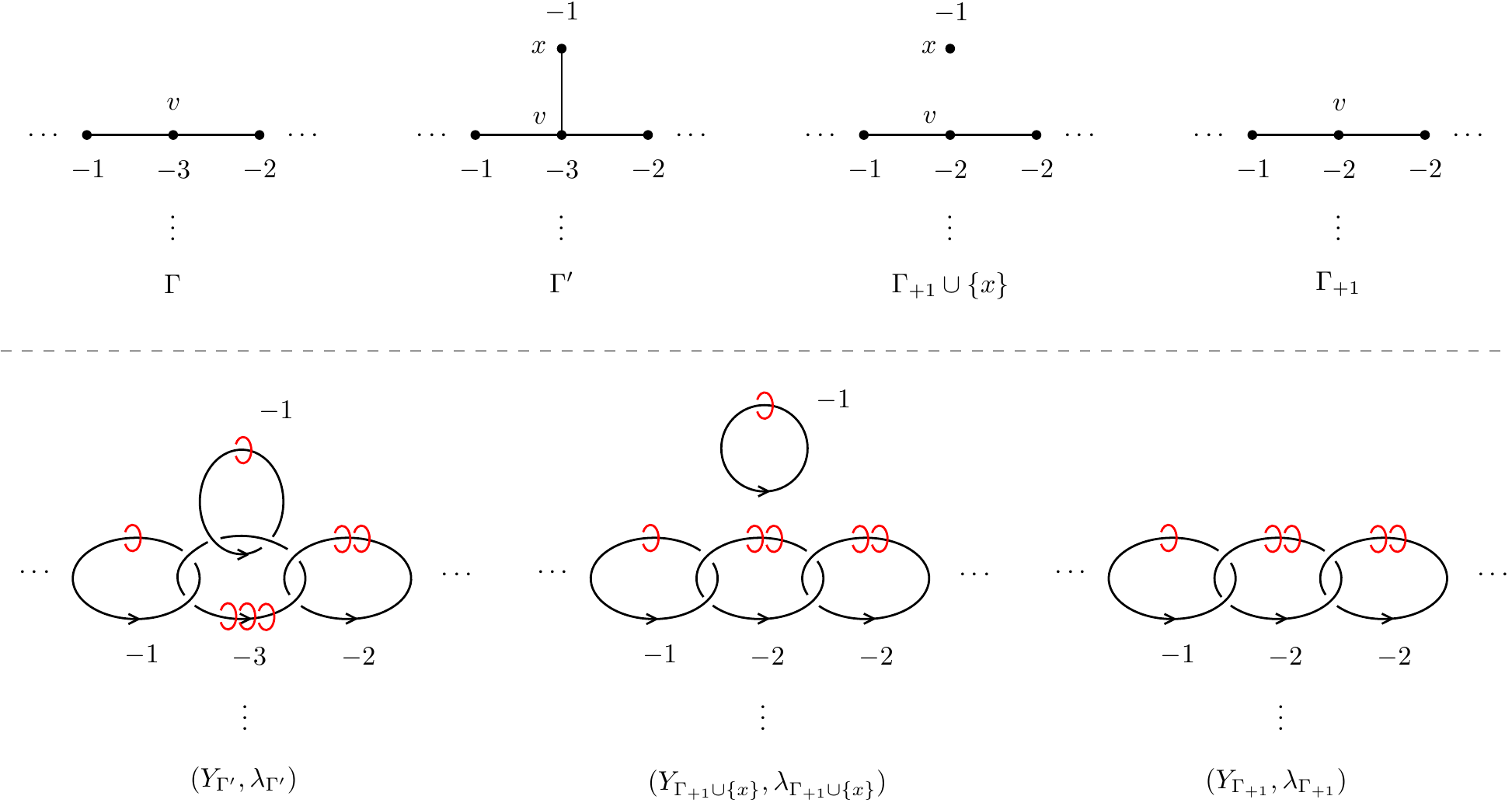}
\caption{Top: plumbing diagrams corresponding to $\Gamma'$, $\Gamma_{+1} \cup \{x\}$, and $\Gamma_{+1}$. Bottom: the multicurves. Each meridian is positively-oriented with respect to the corresponding unknot.}\label{fig:A.1}
\end{figure}

We now define isomorphisms 
\[
\mathbb{F}: \Hl(\Gamma') \rightarrow \Hl(\Gamma_{+1} \cup \{x\}) \text{ and }\ \mathbb{G}:  \Hl(\Gamma_{+1} \cup \{x\})  \rightarrow \Hl(\Gamma_{+1})
\]
between the corresponding lattice homologies. First, let \[\bar{\mathbb{F}}:\Char(\Gamma')\to\Char(\Gamma_{+1} \cup \{x\}) \]  be the map which sends an element $k\in\Char(\Gamma')$  to the characteristic element $k'$ on $\Gamma_{+1} \cup \{x\}$ given by
\[
\ev{k'}{w} =
\begin{cases}
\ev{k}{w} & \text{for } w \neq v \\
\ev{k}{w} - \ev{k}{x} & \text{for } w = v.
\end{cases}
\]
Note that this is just the map induced by the basis change $v \mapsto v - x$, which corresponds to the handleslide inducing the diffeomorphism $W_{\Gamma'} \cong W_{\Gamma_{+1} \cup \{x\}}$. We likewise define a map \[\bar{\mathbb{G}}:\Char(\Gamma_{+1} \cup \{x\})\to \Char(\Gamma_{+1})\] as follows. Let $k$ be an element of $\Char(\Gamma_{+1} \cup \{x\})$. Set
\[
\bar{\mathbb{G}}(k) = 
\begin{cases}
1 \otimes k|_{W_{\Gamma_{+1}}} & \text{if } \ev{k}{x} = 1 \\
(-1) \otimes k|_{W_{\Gamma_{+1}}} & \text{if } \ev{k}{x} = - 1 \\
0 & \text{otherwise},
\end{cases}
\]
where $\smash{\Gamma_{+1}}$ is viewed as a subgraph of $\smash{\Gamma_{+1} \cup \{x\}}$ in the obvious way, and $\smash{k|_{W_{\Gamma_{+1}}}}$ is the restriction of $k$ to the span of vertices in this subgraph.

\begin{lemma}\label{lem:A.1}
The maps $\bar{\mathbb{F}}$ and $\bar{\mathbb{G}}$ descend to well-defined maps
\[
\mathbb{F}: \Hl(\Gamma') \rightarrow \Hl(\Gamma_{+1} \cup \{x\}) \text{ and }\ \mathbb{G}:  \Hl(\Gamma_{+1} \cup \{x\})  \rightarrow \Hl(\Gamma_{+1}).
\]
\end{lemma}
\begin{proof}
We show that $\bar{\mathbb{F}}$ respects the relations of Definition~\ref{def:3.1}. Let $k$ be an element of $\Char(\Gamma')$ and let $u$ be a fixed vertex in $\Gamma'$ with framing $-q$. Suppose that $k \sim 0$ due to a relation of Type I associated to $u$. If $u \neq v$, then $\ev{k'}{u} = \ev{k}{u}$ and the framing of $u$ in $\Gamma_{+1} \cup \{x\}$ is still $-q$. Hence, we clearly have $k' \sim 0$. Thus, suppose $u = v$. If $\ev{k}{x} \neq \pm 1$ then we are in the previous case, so we may assume that $\ev{k'}{u} = \ev{k}{u} \mp 1$. Since the framing of $v$ changes by one from $\Gamma'$ to $\Gamma_{+1} \cup \{x\}$, this easily implies that once again $k' \sim 0$. Hence $\bar{\mathbb{F}}$ respects relations of Type I.

Now consider relations of Type II. Suppose $\ev{k}{u} = -q$, and write $
s = k - 2(u, -)_{\Gamma'}$. Then $k \sim (-1)^q \otimes s$, so we must show that 
\[
\bar{\mathbb{F}}(k) \sim (-1)^{q} \otimes \bar{\mathbb{F}}(s). 
\]
We subdivide into several cases. First suppose that $u \neq v$. Then the left-hand side of the desired relation is
\begin{equation}\label{eq:A.1}
k' \sim (-1)^q \otimes (k' - 2(u, -)_{\Gamma_{+1} \cup \{x\}})
\end{equation}
by a relation of Type II associated to $u$. We claim the characteristic elements $k' - 2(u, -)_{\Gamma_{+1} \cup \{x\}}$ and $s'$ are equal. Indeed, taking their evaluation on an arbitrary vertex $w$, we have
\begin{equation}\label{eq:A.2}
\ev{k'}{w} - 2(u, w)_{\Gamma_{+1} \cup \{x\}} =
\begin{cases}
\ev{k}{w} - 2(u, w)_{\Gamma_{+1} \cup \{x\}} & \text{for } w \neq v \\
\ev{k}{w} - \ev{k}{x} - 2(u, w)_{\Gamma_{+1} \cup \{x\}} & \text{for } w = v
\end{cases}
\end{equation}
and
\begin{align}\label{eq:A.3}
\ev{s'}{w} &= 
\begin{cases}
\ev{s}{w} & \text{for } w \neq v \\
\ev{s}{w} - \ev{s}{x} & \text{for } w = v
\end{cases} \nonumber \\
&=
\begin{cases}
\ev{k}{w} - 2(u, w)_{\Gamma'} & \text{for } w \neq v \\
(\ev{k}{w} - 2(u, w)_{\Gamma'}) - (\ev{k}{x} - 2(u, x)_{\Gamma'}) & \text{for } w = v.
\end{cases}
\end{align}
If $u \neq x$ (in addition to $u \neq v$), then $\smash{(u, -)_{\Gamma_{+1} \cup \{x\}} = (u, -)_{\Gamma'}}$ and $(u, x)_{\Gamma'} = 0$. Comparing \eqref{eq:A.2} and \eqref{eq:A.3} then gives the claim. If on the other hand $u = x$, then the identity $\smash{(u, w)_{\Gamma_{+1} \cup \{x\}} = (u, w)_{\Gamma'}}$ holds as long as $w \neq v$. In this case, it remains to explicitly check
\[
\ev{k'}{v} - 2(x, v)_{\Gamma_{+1} \cup \{x\}} = \ev{k}{v} - \ev{k}{x} - 0 = \ev{k}{v} - \ev{k}{x}
\]
and
\begin{align*}
\ev{s'}{v} &= (\ev{k}{v} - 2(x, v)_{\Gamma'}) - (\ev{k}{x} - 2(x,x)_{\Gamma'}) \\
&= \ev{k}{v} + 2 - \ev{k}{x} -2 \\
&= \ev{k}{v} - \ev{k}{x}.
\end{align*}
Hence in both situations, $\smash{k' - 2(u, -)_{\Gamma_{+1} \cup \{x\}}}$ and $s'$ are equal. Combining this with \eqref{eq:A.1} gives the desired relation between $\bar{\mathbb{F}}(k)$ and $\bar{\mathbb{F}}(s)$ in the case that $u \neq v$.

Now suppose $u = v$, so that $\ev{k}{v} = -p$. As in the first paragraph, we may assume $\ev{k}{x} = \pm 1$. Moreover, if $\ev{k}{x} = +1$, then we have $k \sim k + 2(x, -)_{\Gamma'}$ via a relation of Type II associated to $x$. But this new characteristic element is related to zero, since its evaluation on $v$ is $-p - 2$. Hence we may assume $\ev{k}{x} = -1$. Then $\ev{k'}{v} = \ev{k}{v} - \ev{k}{x} = -p + 1$, which is the decoration of $v$ in $\Gamma_{+1} \cup \{x\}$. By a relation of Type II, we thus have
\begin{equation}\label{eq:A.4}
k' \sim (-1)^{p-1} \otimes (k' - 2(v, -)_{\Gamma_{+1} \cup \{x\}}).
\end{equation}
As before, we compare $\smash{k' - 2(v, -)_{\Gamma_{+1} \cup \{x\}}}$ and $s'$. Examining \eqref{eq:A.2} and \eqref{eq:A.3}, it is again easy to see that these agree for all vertices $w$ other than $v$ or $x$. Evaluating on $v$, we check
\[
\ev{k'}{v} - 2(v, v)_{\Gamma_{+1} \cup \{x\}} = p - 1 
\]
and 
\[
\ev{s'}{v} = (\ev{k}{v} - 2(v, v)_{\Gamma'}) - (\ev{k}{x} - 2(v,x)_{\Gamma'}) = p - 1.
\]
However, evaluating on $x$, we have
\[
\ev{k'}{x} - 2(v,x)_{\Gamma_{+1} \cup \{x\}} = \ev{k}{x} - 0 = -1
\]
while
\[
\ev{s'}{x} = \ev{s}{x} = \ev{k}{x} - 2(u, x)_{\Gamma'} = -1 + 2 = 1.
\]
Hence in this case $k' - 2(v, -)_{\Gamma_{+1} \cup \{x\}}$ and $s'$ are \textit{not} equal, as they differ on $x$. Instead, we have
\begin{equation}\label{eq:A.5}
k' - 2(v,-)_{\Gamma_{+1} \cup \{x\}} \sim (-1) \otimes s'
\end{equation}
via a relation of Type II associated to $x$. (Here, we are using the fact that $x$ is an isolated vertex in $\Gamma_{+1} \cup \{x\}$.) Combining \eqref{eq:A.4} and \eqref{eq:A.5} gives the desired relation between $\bar{\mathbb{F}}(k)$ and $\bar{\mathbb{F}}(s)$. 

The case when $\ev{k}{u} = q$ is analogous. This shows that $\bar{\mathbb{F}}$ descends to a map between the relevant lattice homologies. The result for $\bar{\mathbb{G}}$ is immediate.
\end{proof}

\begin{lemma}\label{lem:A.2}
The maps $\mathbb{F}$ and $\mathbb{G}$ are isomorphisms, and the diagrams
\begin{center}
\begin{tikzpicture}[scale=1]
\node (A) at (0,1.75) {$\Is(Y_{\Gamma'}, \lambda_{\Gamma'}) \ \cong$};
\node (B) at (3.35,1.72) {$\Is(Y_{\Gamma_{+1} \cup \{x\}}, \lambda_{\Gamma_{+1} \cup \{x\}})$};
\node (F) at (0,0) {$\Hl(\Gamma')$};
\node (G) at (3.35,0) {$\Hl(\Gamma_{+1} \cup \{x\})$};
\path[->,font=\scriptsize,>=angle 90]

(F) edge node[left]{$\Ts$} (A)
(G) edge node[left]{$\Ts$} (B)
(F) edge node[above]{$\mathbb{F}$} (G);
\end{tikzpicture}
\hspace{-0.2cm}
\begin{tikzpicture}[scale=1]
\node (A) at (0,0) {};
\node (B) at (0,1) {\text{ and }};
\end{tikzpicture}
\hspace{-0.2cm}
\begin{tikzpicture}[scale=1]
\node (B) at (0,1.72) {$\Is(Y_{\Gamma_{+1} \cup \{x\}}, \lambda_{\Gamma_{+1} \cup \{x\}}) \ \cong$};
\node (C) at (3.6,1.74) {$\Is(Y_{\Gamma_{+1}}, \lambda_{\Gamma_{+1}})$};
\node (G) at (0,0) {$\Hl(\Gamma_{+1} \cup \{x\})$};
\node (H) at (3.6,0) {$\Hl(\Gamma_{+1})$};
\path[->,font=\scriptsize,>=angle 90]

(G) edge node[left]{$\Ts$} (B)
(H) edge node[left]{$\Ts$} (C)
(G) edge node[above]{$\mathbb{G}$} (H);
\end{tikzpicture}
\end{center}
commute.
\end{lemma}

\begin{proof}
It is straightforward to construct an inverse to $\mathbb{F}$ on the level of characteristic elements; this sends $k \in \Char(\Gamma_{+1} \cup \{x\})$ to the element $k' \in \Char(\Gamma')$ defined by
\[
\ev{k'}{w} =
\begin{cases}
\ev{k}{w} & \text{for } w \neq v \\
\ev{k}{w} + \ev{k}{x} & \text{for } w = v.
\end{cases}
\]
The proof that this descends to a map on lattice homology is similar to that of Lemma~\ref{lem:A.1}. The fact that the left-hand square commutes is a consequence of the following two facts:
\begin{enumerate}
\item The basis change $v \mapsto v - x$ corresponds to the handleslide inducing the diffeomorphism $W_{\Gamma'} \cong W_{\Gamma_{+1} \cup \{x\}}$; and,
\item The disk systems $D_{\Gamma'}$ and $\smash{D_{\Gamma_{+1} \cup \{x\}}}$ coincide under the identification of $\Is(Y_{\Gamma'}, \lambda_{\Gamma'})$ with $\smash{\Is(Y_{\Gamma_{+1} \cup \{x\}}, \lambda_{\Gamma_{+1} \cup \{x\}})}$ afforded by Remark~\ref{rem:tautological}.
\end{enumerate}
It is likewise straightforward to verify that $\mathbb{G}$ is an isomorphism. On the level of characteristic elements, the inverse is given by sending $k \in \Char(\Gamma_{+1})$ to $(1/2) \otimes k^+ + (-1/2) \otimes k^-$, where $k^{\pm} \in \Char(\Gamma_{+1} \cup \{x\})$ are defined by
\[
\ev{k^\pm}{w} =
\begin{cases}
\ev{k}{w} & \text{for } w \neq x \\
\pm 1 & \text{for } w = x.
\end{cases}
\]
To see that the right-hand square commutes, observe that $W_{\Gamma_{+1} \cup \{x\}}$ is the blow-up of $W_{\Gamma_{+1}}$. Moreover, under the isomorphism
\[\smash{\Is(Y_{\Gamma_{+1} \cup \{x\}}, \lambda_{\Gamma_{+1} \cup \{x\}})} \cong \smash{\Is(Y_{\Gamma_{+1}}, \lambda_{\Gamma_{+1}})},\] the disk system $\smash{D_{\Gamma_{+1} \cup \{x\}}}$ is homologous to the standard disk system $D_{\Gamma_{+1}}$ coming from $W_{\Gamma_{+1}}$, union a single copy of the exceptional divisor. Applying the sign relation of Theorem \ref{thm:decomposition} and the blow-up formula then gives the claim.
\end{proof}


\bibliographystyle{amsalpha}
\bibliography{bib}

\end{document}